\documentclass[12pt,reqno]{amsart}

\usepackage{amscd} 
\usepackage{amsfonts} 
\usepackage{amsmath} 
\usepackage{amsrefs} 
\usepackage{amssymb} 
\usepackage{amsthm} 
\usepackage{latexsym} 
\usepackage{mathrsfs} 
\usepackage{mathtools} 
\mathtoolsset{showonlyrefs} 
\usepackage{slashed}  

\usepackage{verbatim}
\usepackage{cases}
\usepackage[dvips]{epsfig}
\usepackage{epsf} 
\usepackage[bookmarksnumbered,pdfpagelabels=true,plainpages=false,colorlinks=true,
            linkcolor=black,citecolor=black,urlcolor=black]{hyperref}
\usepackage{url}
\usepackage{color}
\usepackage{xcolor}
\usepackage{graphicx}
\usepackage{enumitem} 

\usepackage{pifont} 
\usepackage{upgreek} 
\usepackage{fancyhdr} 
\usepackage{calligra} 
\usepackage{marvosym} 
\usepackage[percent]{overpic} 
\usepackage{pict2e} 
\usepackage[hypcap=false]{caption} 
\usepackage{wasysym} 
\usepackage{accents}

\usepackage[papersize={8.5in,11in}, margin=1in,footskip=0.4in,vmarginratio={1:1},hmarginratio={1:1}]{geometry} 
\newcommand{\la}{\langle}
\newcommand{\ra}{\rangle}

\newcommand{\cl}{\mathcal}

\theoremstyle{plain}
\newtheorem{theorem}{Theorem}[section]

\newtheorem{proposition}[theorem]{Proposition}
\newtheorem{corollary}[theorem]{Corollary}

\theoremstyle{definition}
\newtheorem{remark}[theorem]{Remark}

\newtheorem{definition}[theorem]{Definition}

\numberwithin{equation}{section}
\allowdisplaybreaks

\newcommand{\norm}[1]{\|#1\|}

\newcommand{\mss}{\hspace{0.2cm}}
\newcommand{\ms}{\hspace{0.25cm}}

\newcommand{\cA}{\mathcal{A}}
\newcommand{\cB}{\mathcal{B}}
\newcommand{\cC}{\mathcal{C}}
\newcommand{\cF}{\mathcal{F}}
\newcommand{\cL}{\mathcal{L}}
\newcommand{\cR}{\mathcal{R}}
\newcommand{\cS}{\mathcal{S}}

\newcommand{\bC}{\mathbb{C}}
\newcommand{\bE}{\mathbb{E}}
\newcommand{\bR}{\mathbb{R}}

\newcommand{\fR}{\mathfrak{R}}

\newcommand{\Proj}{\mathsf{\Pi}}

\newcommand{\ii}{\mathrm{i}}

\begin{document}
\title[Conformal viscous hydrodynamics]{Local well-posedness in Sobolev spaces for first-order
conformal causal relativistic viscous hydrodynamics}
\author[Bemfica, Disconzi, Rodriguez, Shao]{
Fabio S. Bemfica$^{* \$}$, 
Marcelo M. Disconzi$^{** \#}$, 
Casey Rodriguez$^{*** \dagger}$,
and Yuanzhen Shao$^{**** \ddagger}$
}

\thanks{$^{\$}$FSB gratefully acknowledges support from a Discovery grant administered by Vanderbilt University.
}
	
\thanks{$^{\#}$MMD gratefully acknowledges support from a Sloan Research Fellowship 
provided by the Alfred P. Sloan foundation, from NSF grant \# 1812826,
 from a Discovery grant administered by Vanderbilt University, and from
 a Dean's Faculty Fellowship.
}


\thanks{$^{*}$Universidade Federal do Rio Grande do Norte, Natal, RN, Brazil.
\texttt{fabio.bemfica@ect.ufrn.br}}

\thanks{$^{**}$Vanderbilt University, Nashville, TN, USA.
\texttt{marcelo.disconzi@vanderbilt.edu}}

\thanks{$^{***}$Massachusetts Institute of Technology, Cambridge, MA, USA.
\texttt{caseyrod@mit.edu}}

\thanks{$^{****}$The University of Alabama, Tuscaloosa, AL, USA.
\texttt{yshao8@ua.edu}}

\begin{abstract}
In this manuscript, we study the theory of conformal relativistic 
viscous hydrodynamics introduced in \cite{BemficaDisconziNoronha}, which
provided a causal and stable first-order theory of relativistic fluids with viscosity.
The local well-posedness of its equations of motion has been previously established
in Gevrey spaces. Here, we improve this result by proving local well-posedness
in Sobolev spaces.

\bigskip

\noindent \textbf{Keywords:} relativistic viscous fluids; conformal symmetry;
causality; local well-posedness.

\bigskip

\noindent \textbf{Mathematics Subject Classification (2010):} 
Primary: 35Q75; 
Secondary: 	
	35Q35,  	
35Q31, 

\end{abstract}

\maketitle

\tableofcontents

\section{Introduction\label{S:Intro}}
Relativistic hydrodynamics is an essential tool in several branches of physics,
including high-energy nuclear physics \cite{Baier:2007ix}, astrophysics \cite{RezzollaZanottiBookRelHydro}, and cosmology \cite{WeinbergCosmology},
and it is also a fertile source of mathematical problems  (see, e.g., the monographs
\cite{ChristodoulouShocks,ChristodoulouShockDevelopment,ChoquetBruhatGRBook,
AnileBook,RezzollaZanottiBookRelHydro} and references therein). This paper
is concerned with the local well-posedness of the Cauchy 
problem to the equations of motion of relativistic viscous fluids.

More precisely, we consider the energy-momentum tensor for a relativistic conformal
fluid given by
\begin{align}
\label{E:Energy_momentum}
\mathcal{T}_{\alpha \beta} & = 
(\varepsilon + A )(u_\alpha u_\beta +\frac{1}{3} \Proj_{\alpha \beta} )
- \upeta \upsigma_{\alpha \beta} + u_\alpha Q_\beta + u_\beta Q_\alpha,
\end{align}
where
\begin{align}
\begin{split}
A & 
= 3 \upchi (\frac{ 1}{\uptheta} u^\mu \nabla_\mu \uptheta + \frac{1}{3} \nabla_\mu u^\mu ),
\\
Q_\alpha & = \uplambda  ( \frac{1}{\uptheta} \Proj_{\alpha}^\mu  \nabla_\mu \uptheta + 
u^\mu \nabla_\mu u_\alpha ),
\\
\upsigma_{\alpha\beta} & = \Proj_\alpha^\mu \nabla_\mu u_\beta + \Proj_\beta^\mu \nabla_\mu u_\alpha
-\frac{2}{3} \Proj_{\alpha\beta} \nabla_\mu u^\mu.
\end{split}
\nonumber
\end{align}
Above, $\varepsilon$ is the fluid's energy density;
$u$ is the fluid's four-velocity, which satisfies the constraint
\begin{align}
\label{E:u_unit}
g_{\alpha\beta} u^\alpha u^\beta &= -1,
\end{align}
where $g$ is the spacetime metric\footnote{By ``metric" we always mean a ``Lorentzian
metric."}; $\Proj$ is the projection onto the space orthogonal to $u$, given by
$\Proj_{\alpha\beta} = g_{\alpha\beta}+ u_\alpha u_\beta$;
 $\uptheta$ is the temperature that satisfies $\varepsilon = \varepsilon_0 \uptheta^4 $,
 where  $\varepsilon_0 > 0$ a constant; $\upeta$, $\upchi$, and $\uplambda$ are
 transport coefficients, which are known functions of $\varepsilon$ and 
 model the viscous effects in the fluid; and
 $\nabla$ is the covariant derivative associated with the metric $g$.
Indices are raised and lowered using the spacetime metric, lowercase Greek indices vary from $0$
to $3$, Latin indices vary from $1$ to $3$, 
repeated indices are summed over their range, and expressions such as
$z_{\alpha}$, $w_{\alpha\beta}$, etc. represent the components of a vector or tensor
with respect to a system of coordinates $\{ x^\alpha \}_{\alpha=0}^3$ in spacetime,
where the coordinates are always chosen so that $x^0=t$ represents a time coordinate.
We will consider the fluid dynamics in a fixed background, so that the metric $g$ is given.

The equations of motion are given by 
\begin{align}
\label{E:Div_T}
\nabla_\alpha \mathcal{T}^\alpha_\beta = 0
\end{align}
supplemented by the constraint \eqref{E:u_unit}. 

We now state our result. After the statement, we discuss our assumptions and provide
some further context.
We note that in view of \eqref{E:u_unit}, 
it suffices to provide the components of $u$ tangent to $\{t=0\}$
as initial data; this explains the statement involving the projector $\mathcal{P}$
in the Theorem.

\begin{theorem}
\label{T:main_theorem}
Let $g$ be the Minkowski metric on $\mathbb{R} \times \mathbb{T}^3$,
where $\mathbb{T}^3$ is the three-dimensional torus. Let $\upeta: (0,\infty) \rightarrow
(0,\infty)$ be an analytic function, $\upchi = a_1 \upeta$, and $\uplambda = a_2 \upeta$,
where $a_1$ and $a_2$ are positive constants satisfying 
$a_1 > 4$ and $a_2 \geq 3a_1/(a_1 - 1)$. 
Let $\varepsilon_{(0)} \in H^r(\mathbb{T}^3,\mathbb{R})$,
$\varepsilon_{(1)} \in H^{r-1}(\mathbb{T}^3,\mathbb{R})$,
$u_{(0)} \in H^r(\mathbb{T}^3,\mathbb{R}^3)$, 
and 
$u_{(1)} \in H^{r-1}(\mathbb{T}^3,\mathbb{R}^3)$
be given, where $H^r$ is the Sobolev space and $r > 7/2$. Assume that 
$\varepsilon_{(0)} \geq C > 0$ for some constant $C$.

Then, there exists a $T>0$, a function 
\begin{align}
\nonumber
\varepsilon \in C^0([0,T),H^r(\mathbb{T}^3,\mathbb{R})) \cap 
C^1([0,T),H^{r-1}(\mathbb{T}^3,\mathbb{R})) \cap 
C^2([0,T),H^{r-2}(\mathbb{T}^3,\mathbb{R})),
\end{align}
and a vector field 
\begin{align}
u \in C^0([0,T),H^r(\mathbb{T}^3,\mathbb{R}^4)) \cap 
C^1([0,T),H^{r-1}(\mathbb{T}^3,\mathbb{R}^4)) \cap 
C^2([0,T),H^{r-2}(\mathbb{T}^3,\mathbb{R}^4))
\end{align}
such that equations \eqref{E:u_unit} and \eqref{E:Div_T} hold on 
$[0,T)\times \mathbb{T}^3$, and satisfy
$\varepsilon(0,\cdot) = \varepsilon_{(0)}$, 
$\partial_t \varepsilon(0,\cdot) = \varepsilon_{(1)}$,
$\mathcal{P}u(0,\cdot) = u_{(0)}$,
and $\mathcal{P} \partial_t u(0,\cdot) = u_{(1)}$, where $\partial_t$
is the derivative with respect to the first coordinate in $[0,T)\times \mathbb{T}^3$
and $\mathcal{P}$ is the canonical projection from the tangent bundle of
$[0,T)\times \mathbb{T}^3$ onto the tangent bundle of $\mathbb{T}^3$.
Moreover, $(\varepsilon,u)$ is the unique solution with the stated properties.
\end{theorem}

One of the main challenges in the theory of relativistic viscous hydrodynamics
is to construct physically meaningful theories that respect causality, (linear)
stability, and local well-posedness.
The literature on this topic is vast and we refer the reader
to 
\cite{Heinz:2013th,Baier:2007ix,Disconzi_Kephart_Scherrer_2015,
DisconziKephartScherrerNew,Hiscock_Lindblom_stability_1983,Hiscock_Lindblom_instability_1985,RezzollaZanottiBookRelHydro,Strickland:2014pga,Strickland:2014eua} and references therein for discussion and background. Despite the importance
of relativistic viscous hydrodynamics in the study, for example, of the quark-gluon-plasma
that forms in heavy ion-collisions \cite{Heinz:2013th,Baier:2007ix} or in neutron
star mergers \cite{RezzollaImportanceViscosityNeutronStars}, very few models have been
showed to be causal, stable, and locally well-posed, and typical
results of this nature have been only partial \cite{DisconziViscousFluidsNonlinearity,DisconziCzubakNonzero,
BemficaDisconziNoronha_IS_bulk,GerochLindblomCausal,Lehner:2017yes}.

The energy-momentum \eqref{E:Energy_momentum} was introduced in
\cite{BemficaDisconziNoronha}, where a new approach to the formulation of relativistic
viscous hydrodynamics was proposed for the case of a conformal fluid. 
The equations of motion derived from 
\eqref{E:Energy_momentum}, i.e., \eqref{E:u_unit} and \eqref{E:Div_T},
were showed to be causal, linearly stable, and locally well-posed in
Gevrey spaces in \cite{BemficaDisconziNoronha,DisconziFollowupBemficaNoronha}.
In this work, we extend these results by establishing local well-posedness
in Sobolev spaces\footnote{However, only the case of a fixed background Minkowski
metric is treated here, whereas in \cite{BemficaDisconziNoronha,DisconziFollowupBemficaNoronha} the coupling of \eqref{E:Energy_momentum} with Einstein's equations has been studied.}.

Conformal
fluids satisfy the property that the ratio between any two transport coefficients
is constant \cite{Bhattacharyya:2008jc,Baier:2007ix}, which explains our
assumptions on $\upchi$ and $\upeta$. The hypothesis on $a_1$ and $a_2$, in turn,
are the same as in \cite{BemficaDisconziNoronha,DisconziFollowupBemficaNoronha}, and
ensure the causality and linear stability of solutions.

We work on $\mathbb{T}^3$ for simplicity, since using the domain of dependence
property (proved in \cite{DisconziFollowupBemficaNoronha}) one can adapt the proof
to $\mathbb{R}^3$. The assumption $\varepsilon_0 \geq C > 0$, on the other 
hand, is crucial. Without it the equations can degenerate, resulting in a free-boundary
dynamics, a problem that remains largely open even in the case
of a relativistic perfect fluid
\cite{MR2813883,MR3182488,HadzicShkollerSpeck,JangLeFlochMasmoudi,
SchlueHardStars,Oliynyk,OliynykRelativisticBodiesI,GinsbergRelEulerApriori,
OliynykRelativisticLiquidBodies,DisconziRemarksEinsteinEuler}

\section{A new system of equations}
In this section we derive a new system of equations that will allow
us to establish Theorem \ref{T:main_theorem}. In order to do so, throughout this section,
we assume to be given a sufficiently regular solution to 
\eqref{E:u_unit}-\eqref{E:Div_T}.

Using \eqref{E:u_unit} to decompose $\nabla_\alpha \mathcal{T}^\alpha_\beta$
in the directions parallel and orthogonal to $u$, we can rewrite \eqref{E:Div_T}
as
\begin{subequations}{\label{E:EofM}}
\begin{align}
&
u^\alpha\nabla_\alpha A+\frac{4}{3}A\nabla_\alpha u^\alpha
+\nabla_\alpha Q^\alpha+Q_\alpha u^\lambda\nabla_\lambda u^\alpha
-\frac{1}{2}\upeta \upsigma^{\mu\nu}\upsigma_{\mu\nu}+
\frac{4}{3\upchi} \varepsilon_0 \uptheta^4
A=0,
\label{E:EofM_1}
\\
& 
\frac{1}{3}\Proj^\alpha_\mu\nabla_\alpha A+\frac{4}{3}A u^\alpha\nabla_\alpha u_\mu-\upeta\nabla_\alpha\upsigma^\alpha_\mu
+\frac{\upeta}{2} \upsigma^{\alpha\beta}\upsigma_{\alpha\beta} u_\mu+
3\upeta\upsigma_{\mu\lambda} u^\alpha\nabla_\alpha u^\lambda
+u^\alpha\nabla_\alpha Q_\mu
\\
&
-u_\mu Q^\lambda u^\alpha\nabla_\alpha u_\lambda
+\nabla_\alpha u^\alpha  Q_\mu 
+ Q^\alpha\nabla_\alpha u_\mu +\frac{4\varepsilon}{3\uplambda} Q_\mu-\frac{3\upeta}{\uplambda}\upsigma_{\mu\nu} Q^\nu=0.
\label{E:EofM_2}
\end{align}
\end{subequations}
Introducing 
\begin{align}
\begin{split}
S_\alpha^{\mss \beta} &= \Proj_\alpha^\mu \nabla_\mu u^\beta,
\\
S^\alpha &= u^\mu \nabla_\mu u^\alpha, 
\end{split}
\nonumber
\end{align}
we find 
\begin{subequations}{\label{E:Equations_1st_order}}
\begin{align}
u^\mu \nabla_\mu A + \nabla_\mu Q^\mu + r_1 & = 0,
\label{E:Equations_1st_order_1}
\\
\Proj^{\alpha \mu}\nabla_\mu A +3 u^\mu \nabla_\mu Q^\alpha 
+ B_\nu^{\mss\alpha \mu\lambda}\nabla_\lambda S_\mu^{\mss \nu}
+ r_2
& =0,
\label{E:Equations_1st_order_2}
\\
-\frac{1}{\upchi} \Proj^{\alpha \mu}\nabla_\mu A + \frac{3}{\uplambda}
u^\mu \nabla_\mu Q^\alpha - 3 u^\mu \nabla_\mu S^\alpha
+ \Proj^{\alpha\mu}\nabla_\mu S_\nu^{\mss \nu} +r_3& = 0,
\label{E:Equations_1st_order_3}
\\
u^\mu \nabla_\mu S_\alpha^{\mss \beta} - \Proj_\alpha^\nu \nabla_\nu
S^\beta + r_4 & = 0,
\label{E:Equations_1st_order_4}
\\
\frac{1}{\uptheta} u^\mu \nabla_\mu \uptheta + \frac{1}{3} \nabla_\mu u^\mu +
r_5 & = 0,
\label{E:Equations_1st_order_5}
\\
\frac{1}{\uptheta} \Proj^{\alpha \mu} \nabla_\mu \uptheta + u^\mu \nabla_\mu u^\alpha
+ r_6 & = 0,
\label{E:Equations_1st_order_6}
\end{align}
\end{subequations}
where 
\begin{align}
B_\nu^{\mss\alpha \mu\lambda} & = 
-3\upeta(\updelta_\nu^\alpha \Proj^{\mu \lambda}
+ \updelta_\nu^\lambda \Proj^{\alpha \mu} 
- \frac{2}{3} \updelta_\nu^\mu \Proj^{\alpha \lambda}), 
\nonumber
\end{align}
and $r_i$, $i=1,\dots,6$ are smooth functions of 
$A$, $Q^\alpha$, $S^\alpha$, $S_\alpha^{\mss \beta}$, 
$\uptheta$, and $u^\alpha$; no derivative of such quantities appears in the 
$r_i$'s. Above and throughout, $\updelta$ is the Kronecker delta.

The derivation of \eqref{E:Equations_1st_order} is as follows: 
equations \eqref{E:Equations_1st_order_1} and 
\eqref{E:Equations_1st_order_2}
are equations \eqref{E:EofM_1} and \eqref{E:EofM_2}, respectively;
equations
\eqref{E:Equations_1st_order_5} and \eqref{E:Equations_1st_order_6}
are simply the definition of $A$ and $Q^\alpha$;
equations \eqref{E:Equations_1st_order_3} and 
\eqref{E:Equations_1st_order_4} follow from contracting the identities
\begin{align}
\begin{split}
\nabla_\mu \nabla_\nu \uptheta - \nabla_\nu \nabla_\mu \uptheta & = 0,\\
\nabla_\mu \nabla_\nu u^\alpha -\nabla_\nu \nabla_\mu u^\alpha & = 
R_{\mu \nu \ms \lambda}^{\mss \mss \alpha} u^\lambda = 0,
\end{split}
\nonumber
\end{align}
with $u^\mu$ and then with $\Proj^\nu_\lambda$. We also used the identities
\begin{align}
\begin{split}
\frac{1}{\uptheta} \nabla_\alpha \uptheta & = 
- \frac{1}{3 \upchi} u_\alpha A + \frac{1}{\uplambda} Q_\alpha
+\frac{1}{3} u_\alpha S_\mu^{\mss \mu} - \Proj_{\alpha \mu} S^\mu,
\\
\nabla_\alpha u^\beta & = -u_\alpha S^\beta + S_\alpha^{\mss \beta}.
\end{split}
\nonumber
\end{align} 
We write equations \eqref{E:Equations_1st_order} as a quasilinear first order system
for the the variable $\mathbf{\Psi} = (A, Q^\alpha,S^\alpha, 
S_0^{\mss \alpha},S_1^{\mss \alpha},S_2^{\mss \alpha},S_3^{\mss \alpha},
\uptheta, u^\alpha)^T$, with ${}^T$ being the transpose,  as
\begin{align}
\mathcal{A}^\alpha \nabla_\alpha \mathbf{\Psi} + \mathcal{R} & = 0,
\label{E:Matrix_system_1st_order}
\end{align}
where $\mathcal{R} = (r_1,\dots,r_6)^T$ and
$\mathcal{A}^\alpha$ is given by
\begin{align}
\mathcal{A}^\alpha=
\begin{bmatrix}
u^\alpha & \updelta^\alpha_\nu & 0_{1\times 4} & 0_{1\times 4} & 0_{1\times 4} & 0_{1\times 4} & 0_{1\times 4} & 0 & 0_{1\times 4}\\
\Proj^{\mu\alpha} & 3u^\alpha I_4 & 0_{4\times 4} & B_\nu^{\mss\mu 0 \alpha} & 
B_\nu^{\mss\mu 1 \alpha} & B_\nu^{\mss\mu 2 \alpha} & B_\nu^{\mss\mu 3 \alpha} & 0_{4\times 1} & 0_{4\times 4} \\
-\frac{\Proj^{\mu\alpha}}{\upchi} & \frac{3u^\alpha}{\uplambda}I_4 & -3u^\alpha I_4 & \updelta^0_\nu \Proj^{\mu\alpha} & \updelta^1_\nu \Proj^{\mu\alpha} & \updelta^2_\nu \Proj^{\mu\alpha} & \updelta^3_\nu \Proj^{\mu\alpha} & 0_{4\times 1} & 0_{4\times 4} \\
0_{4\times 1} & 0_{4\times 4} & -\Proj^\alpha_0 I_4 & u^\alpha I_4 & 0_{4\times 4} & 0_{4\times 4} & 0_{4\times 4} & 0_{4\times 1} & 0_{4\times 4} \\
0_{4\times 1} & 0_{4\times 4} & -\Proj^\alpha_1 I_4 & 0_{4\times 4} & u^\alpha I_4 & 0_{4\times 4} & 0_{4\times 4} & 0_{4\times 1} & 0_{4\times 4} \\
0_{4\times 1} & 0_{4\times 4} & -\Proj^\alpha_2 I_4 & 0_{4\times 4} & 0_{4\times 4} & u^\alpha I_4 & 0_{4\times 4} & 0_{4\times 1} & 0_{4\times 4} \\
0_{4\times 1} & 0_{4\times 4} & -\Proj^\alpha_3 I_4 & 0_{4\times 4} & 0_{4\times 4} & 0_{4\times 4} & u^\alpha I_4 & 0_{4\times 1} & 0_{4\times 4} \\
0 & 0_{1\times 4} &  0_{1\times 4}  & 0_{1\times 4} & 0_{1\times 4} & 0_{1\times 4} &  0_{1\times 4}  & \frac{u^\alpha}{\uptheta} & \frac{\updelta^\alpha_\nu}{3} \\
0_{4\times 1} & 0_{4\times 4} & 0_{4\times 4} & 0_{4\times 4} & 0_{4\times 4} & 0_{4\times 4} & 0_{4\times 4} & \frac{\Proj^{\mu\alpha}}{\uptheta} & I_4\,u^\alpha \\
\end{bmatrix}
.
\nonumber
\end{align}
Equation \eqref{E:Matrix_system_1st_order} is the main equation we will use to derive estimates.

\section{Diagonalization\label{S:Diagonalization}}
Here, we show that under assumptions consistent with 
those of Theorem \ref{T:main_theorem}, we can diagonalize the principal part of
\eqref{E:Matrix_system_1st_order}.

\begin{proposition} Let $\upxi$ be a timelike vector and assume that 
$\upchi > 4 \upeta > 0$ and that $\uplambda \ge \frac{3\upchi \upeta}{\upchi - \upeta}$. Then:\\
\noindent (i) $\det(\mathcal{A}^\alpha \upxi_\alpha) \neq 0$;\\
\noindent (ii) For any spacelike vector $\upzeta$, the eigenvalue problem 
$\mathcal{A}^\alpha(\upzeta_\alpha + \Lambda \upxi_\alpha)  V = 0$
has only real eigenvalues $\Lambda$ and a complete set of eigenvectors $V$.
\end{proposition}
\begin{remark} In practice we will take $\upxi = (1,0,0,0)$ and $\upzeta = (0,\upzeta_1, \upzeta_2, \upzeta_3)$. We note that the assumptions on $\upchi$, $\uplambda$, and $\upeta$ on 
Theorem \ref{T:main_theorem} imply the assumptions on these coefficients in the Proposition.
\end{remark}
\begin{proof}
Let $a$ and $b$ be the projection of $\upzeta + \Lambda \upxi$ on the direction orthogonal and
parallel to $u$, i.e., $a^\alpha = \Proj^{\alpha\mu} (\upzeta_\mu + \Lambda \upxi_\mu)$
and $b = (\upzeta_\alpha + \Lambda \upxi_\alpha)u^\alpha$. Then
\begin{align}
\begin{split}
a^\mu a_\mu & = \Proj^{\alpha \mu} \Proj_{\alpha \nu} (\upzeta_\mu + \Lambda \upxi_\mu)
(\upzeta^\nu + \Lambda \upxi^\nu)
= (g^{\alpha \mu} + u^\alpha u^\mu )(\upzeta_\mu + \Lambda \upxi_\mu)
(g_{\alpha\nu} + u_\alpha u_\nu)(\upzeta^\nu + \Lambda \upxi^\nu)
\\
& = (\upzeta_\alpha + \Lambda \upxi_\alpha)(\upzeta^\alpha + \Lambda \upxi^\alpha)
+b^2.
\end{split}
\nonumber
\end{align}
To simplify the notation, set $\Xi_\alpha = \upzeta_\alpha + \Lambda \upxi_\alpha$.
Then
\begin{align}
\det(\Xi_\alpha \mathcal{A}^\alpha) =
& \det
\begin{bmatrix}
b & \Xi^T & 0_{1\times 4} & 0_{1\times 4} & 0_{1\times 4} & 0_{1\times 4} & 0_{1\times 4} & 0 & 0_{1\times 4}\\
a & 3b I_4 & 0_{4\times 4} & D^{\mss\mu 0 }_\nu & D^{\mss\mu 1 }_\nu & D^{\mss\mu 2 }_\nu & D^{\mss\mu 3 }_\nu & 0_{4\times 1} & 0_{4\times 4} \\
-\frac{a}{\upchi} & \frac{3b}{\uplambda}I_4 & -3b I_4 & \updelta^0_\nu a & \updelta^1_\nu a & \updelta^2_\nu a & \updelta^3_\nu a & 0_{4\times 1} & 0_{4\times 4} \\
0_{4\times 1} & 0_{4\times 4} & -a_0 I_4 & u^\alpha I_4 & 0_{4\times 4} & 0_{4\times 4} & 0_{4\times 4} & 0_{4\times 1} & 0_{4\times 4} \\
0_{4\times 1} & 0_{4\times 4} & -a_1 I_4 & 0_{4\times 4} & b I_4 & 0_{4\times 4} & 0_{4\times 4} & 0_{4\times 1} & 0_{4\times 4} \\
0_{4\times 1} & 0_{4\times 4} & -a_2 I_4 & 0_{4\times 4} & 0_{4\times 4} & b I_4 & 0_{4\times 4} & 0_{4\times 1} & 0_{4\times 4} \\
0_{4\times 1} & 0_{4\times 4} & -a_3 I_4 & 0_{4\times 4} & 0_{4\times 4} & 0_{4\times 4} & b I_4 & 0_{4\times 1} & 0_{4\times 4} \\
0 & 0_{1\times 4} &  0_{1\times 4}  & 0_{1\times 4} & 0_{1\times 4} & 0_{1\times 4} &  0_{1\times 4}  & \frac{b}{\uptheta} & \frac{\Xi^T}{3} \\
0_{4\times 1} & 0_{4\times 4} & 0_{4\times 4} & 0_{4\times 4} & 0_{4\times 4} & 0_{4\times 4} & 0_{4\times 4} & \frac{a}{\uptheta} & bI_4 \\
\end{bmatrix}
\nonumber
\\
& = m_1 m_2,
\nonumber
\end{align}
where we write $\Xi^T$ to emphasize that 
$\Xi^T$ represents a $1\times 4$ piece, and
$D_\nu^{\mss\alpha \mu}  = B_\nu^{\mss\alpha \mu\lambda} \Xi_\lambda$.
$m_2$ is given by
\begin{align}
\begin{split}
m_2 =&
\det
\begin{bmatrix}
\frac{b}{\uptheta} & \frac{\Xi^T}{3} 
\\
 \frac{a}{\uptheta} & bI_4
\end{bmatrix}
\\
& = 
\frac{b^3}{3\uptheta}(3b^2-\Proj_{\alpha\beta}\Xi^\alpha\Xi^\beta),
\end{split}
\nonumber
\end{align}
whereas 
\begin{align}
m_1 &= 
\det
\begin{bmatrix}
b & \Xi^T & 0_{1\times 4} & 0_{1\times 4} & 0_{1\times 4} & 0_{1\times 4} & 0_{1\times 4}\\
a & 3b I_4 & 0_{4\times 4} & D^{\mss\mu 0}_\nu & D^{\mss\mu 1}_\nu & D^{\mss\mu 2}_\nu & D^{\mss\mu 3}_\nu 
\\
-\frac{a}{\upchi} & \frac{3b}{\uplambda}I_4 & -3b I_4 & a \updelta^0_\nu  & a \updelta^1_\nu & a \updelta^2_\nu & a \updelta^3_\nu
 \\
0_{4\times 1} & 0_{4\times 4} & -a_0 I_4 & b I_4 & 0_{4\times 4} & 0_{4\times 4} & 0_{4\times 4}
\\
0_{4\times 1} & 0_{4\times 4} & -a_1 I_4 & 0_{4\times 4} & b I_4 & 0_{4\times 4} & 0_{4\times 4}
\\
0_{4\times 1} & 0_{4\times 4} & -a_2 I_4 & 0_{4\times 4} & 0_{4\times 4} & b I_4 & 0_{4\times 4}
\\
0_{4\times 1} & 0_{4\times 4} & -a_3 I_4 & 0_{4\times 4} & 0_{4\times 4} & 0_{4\times 4} & b I_4
\end{bmatrix}\label{line1}
\\
& = 
b^{9}\det
\begin{bmatrix}
1 & \Xi^T & 0_{1\times 4} \\
a & 3b^2 I_4 & D^{\mss \mu\alpha}_\nu a_\alpha \\
-\frac{a}{\upchi} & \frac{3b^2I_4}{\uplambda} & -3b^2 I_4+a^\mu a_\nu 
\end{bmatrix}\label{line2}
\\
& =
b^9
\det
\begin{bmatrix}
3b^2 I_4 & a & E^{\mu}_\nu \\
\Xi^T & 1 & 0_{1\times 4} \\
\frac{3b^2}{\uplambda}I_4 & -\frac{a}{\upchi} & -3b^2 I_4+a^\mu a_\nu 
\end{bmatrix}\label{line3}
\\
& =
27b^{15}\det
\begin{bmatrix}
3b^2-a^\mu a_\mu & -\Xi_\mu E^\mu_\nu \\
\frac{\uplambda+\upchi}{\uplambda\upchi}a & 3b^2 I_4-a^\mu a_\nu+\frac{E^\mu_\nu}{\uplambda} 
\end{bmatrix}\label{line4}
\\
& =
27b^{15}\det
\begin{bmatrix}
F & d_\nu \\
c^\mu & h^\mu_\nu
\end{bmatrix}=\frac{27b^{15}}{F^3}\det(
Fh^\mu_\nu -c^\mu d_\nu  )\label{line5}
\\
&=\frac{27b^{15}}{F^3}\det(
FG\updelta^\mu_\nu -H_\nu a^\mu )=27b^{15}G^3(
FG -H_\mu a^\mu )\label{line6}
\\
&=27b^{15}G^3(
FG-\frac{\uplambda+\upeta}{\uplambda}F(a^ \mu a_\mu) -\upkappa )\label{line7}.
\end{align}
We now detail how the computations \eqref{line1}-\eqref{line7} were carried
out. These computations made successive use of the formula
\begin{align}
\det 
\begin{bmatrix}
M_1 & M_2 \\
M_3 & M_4
\end{bmatrix}
&=\det(M_1) \det(M_4 - M_3 M_1^{-1} M_2 )\label{id1}\\
&=\det(M_4) \det(M_1 - M_2 M_4^{-1} M_3)\label{id2}
\end{align}
when $M_1^{-1}$ or $M_4^{-1}$ exist, and we defined
\begin{align}
\begin{split}
E^\mu_\nu&=-3\upeta
( a^\alpha a_\alpha \updelta^\mu_\nu+a^\mu\Xi_\nu-\frac{2}{3}a^\mu a_\nu ),
\\
F &= 3b^2 - a^\mu a_\mu,
\\
d_\nu &=-2\upeta a^\alpha a_\alpha (a_\nu-3\Xi_\nu),
\\
c^\mu &=\frac{\uplambda+\upchi}{\uplambda\upchi}a^\mu,
\\
h^\mu_\nu &= 3b^2 \updelta^\mu_\nu-a^\mu a_\nu+\frac{E^\mu_\nu}{\uplambda}
\\
&=3 (b^2-\frac{a^\alpha a_\alpha \upeta}{\uplambda})\updelta^\mu_\nu-\frac{\uplambda-2\upeta}{\uplambda}a^\mu a_\nu
-\frac{3\upeta}{\uplambda}a^\mu\Xi_\nu,
\\
G &= 3(b^2-\frac{ a^\alpha a_\alpha \upeta}{\uplambda}),
\\
H_\nu&=F(\frac{\uplambda-2\upeta}{\uplambda}a_\nu+\frac{3\upeta}{\uplambda}\Xi_\nu)+\frac{\uplambda+\upchi}{\uplambda\upchi}d_\nu,
\\
H_\mu a^\mu&=\frac{\uplambda+\upeta}{\uplambda}F(a^ \mu a_\mu) +\upkappa\\
\upkappa &= c^\mu d_\mu=\frac{4\upeta(\uplambda+\upchi)}{\uplambda\upchi}(a^\alpha a_\alpha)^2.
\end{split}
\label{def1}
\end{align}
From \eqref{line1} to \eqref{line2} we used \eqref{id1} by setting 
\begin{align}
\nonumber
M_1=\begin{bmatrix} b & \Xi^T & 0_{1\times 4} \\
a & 3b I_4 & 0_{4\times 4} \\-\frac{a}{\upchi} & \frac{3b}{\uplambda}I_4 & -3b I_4 
\end{bmatrix}
\end{align}
with $M_2,\,M_3,$ and $M_4$ following accordingly. Although $\det(M_4)=b^{16}$, we multiplied lines 2 to 9 by $b$ and divided column 1  by $b$. Then, the overall multiplicative factor was modified by $b^{16}b^8b^{-1}=b^9$, resulting in \eqref{line2}. After that, we performed the following permutations in \eqref{line2}: the fifth line was brought to the first line after 4 line permutations and the fifth column became the first column after 4 column permutations, obtaining \eqref{line3}, where $E^\mu$ was defined in \eqref{def1}. From \eqref{line3} to \eqref{line4} we made again use of \eqref{id1} by setting $M_1=3b^ 2 I_4$, where $M_2,\,M_3$, and $M_4$ are chosen accordingly. The resulting matrix has the overall factor multiplied by $\det M_1=81b^{8}$, but since we multiplied the first line of the resulting matrix by $3b^ 2$, it reduces to $27 b^6$ and, then, by changing the sign of the last 4 lines, Eq.\ \eqref{line4} is obtained. The first equality in \eqref{line5} corresponds to \eqref{line4} with the definitions that appear in \eqref{def1}. In the second equality it was applied \eqref{id2} with $M_1=F$, where $M_2,\,M_3$, and $M_4$ are chosen accordingly. The $F^{-3}$ factor appears as we multiplied all lines by $F$, then $\det(M_1)F^{-4}=F^{-3}$. The first equality of \eqref{line6} corresponds to the second equality in \eqref{line5} by using the definitions in \eqref{def1}. From the first to the second equality in \eqref{line6}, we used the formula
\begin{align}
\det(A\updelta^ \mu_\nu+\alpha^ \mu\beta_\nu)=A^4+A^3 \alpha^\mu\beta_\mu
\nonumber
\end{align}
with $A=FG$, $\alpha^\mu=a^\mu$, and $\beta_\nu=-H_\nu$. Finally,
\begin{align}
\det(\Xi_\alpha \mathcal{A}^\alpha)=m_1 m_2=
\frac{9b^{18}}{\uptheta}G^3(3b^2 - a^\mu a_\mu)
(FG-F\frac{\uplambda+\upeta}{\uplambda}a^\mu a_\mu -\upkappa ).
\nonumber
\end{align}

We set $\det(\Xi_\alpha \mathcal{A}^\alpha)$ equal to zero to find the eigenvalues and 
eigenvectors. Thus, we need to find the roots $\Lambda$ of
$b = 0$ with multiplicity $18$, $G=0$ (which gives a total of two roots with multiplicity $3$), $3b^2 - a^\mu a_\mu$ 
(which gives a total of $2$ roots with multiplicity 1), and 
$FG-F\frac{\uplambda+\upeta}{\uplambda}a^\mu a_\mu -\upkappa = 0$
(which gives a total of $4$ roots with multiplicity 1), and the corresponding eigenvectors in all cases.

\vskip 0.2cm
\underline{$b^{18}=0$} gives
\begin{align}
\Lambda_1 &= -\frac{u^\alpha \upzeta_\alpha}{u^\beta \upxi_\beta}. 
\nonumber
\end{align}
There are $18$ corresponding linearly independent eigenvectors given by
\begin{align}
\begin{split}
\begin{bmatrix}
0 \\ w_a^\nu \\ 0_{25\times 1}
\end{bmatrix},
\,
\begin{bmatrix}
 0_{26\times 1}\\ w_a^\nu
\end{bmatrix},
\,  
\begin{bmatrix}
\upchi f_\lambda^\lambda \\ 0_{8\times 1}\\ f_0^\nu \\ f_1^\nu \\ f_2^\nu \\ f_3^\nu\\ 0_{5\times 1}
\end{bmatrix},
\end{split}
\nonumber 
\end{align}
where $w_a^\nu=\{w_1^\nu=u^\nu,w_2^\nu,w_3^\nu\}$ are 3 linearly independent 
vectors orthogonal to $\upzeta_\lambda+\Lambda_1 \upxi_\lambda$, and 
$f_\lambda^\nu$ totalizes $16$ components that define the entries in the last vector. However, since these $16$ components are constrained by the $4$ 
equations $\upchi f^\lambda_\lambda a^\mu+D^{\mss \mu\lambda}_\nu f_\lambda^\nu=0$
(where $a^\alpha$ is as above but with $\Lambda=\Lambda_1$), we end up with 12 independent entries. Then, $3+3+12=18$, which equals the multiplicity of the root $\Lambda_1$.

\vskip 0.2cm
\underline{$3b^2 - a^\mu a_\mu=0$} can be written as 
$b^2 - \upbeta  a^\mu a_\mu = 0$, where 
$\upbeta = \frac{1}{3}$. The roots are then $\Lambda_{2,\pm}  =(-u^\mu \upzeta_\mu u^\nu \upxi_\nu+\upbeta\Proj^{\mu\nu}\upxi_\mu\upzeta_\nu \pm\sqrt{\Delta})/((u^\mu \upxi_\mu)^2(1-\upbeta)-\upbeta \upxi^\mu\upxi_\mu)$, where
\begin{align}
\begin{split}
\Delta&=  \upbeta (((u^\mu\upxi_\mu)^2-\Proj^{\mu\nu}\upxi_\mu\upxi_\nu)(\Proj^{\alpha\beta}\upzeta_\alpha\upzeta_\beta-(u^\alpha\upzeta_\alpha)^2)+(u^\mu\upxi_\mu u^\nu\upzeta_\nu+\Proj^{\mu\nu}\upxi_\mu\upzeta_\nu)^2\\
&+(1-\upbeta)(\Proj^{\mu\nu}\upxi_\mu\upxi_\nu \Proj^{\alpha\beta}\upzeta_\alpha\upzeta_\beta-(\Proj^{\mu\nu}\upxi_\mu\upzeta_\nu)^2)).
\end{split}
\nonumber
\end{align} 
We note that these roots are always real when $0 < \upbeta < 1$ because $\Proj_{\alpha\beta}\upxi^\alpha\upxi^\beta<(\upxi_\alpha u^\alpha)^2$, $\Proj_{\alpha\beta}\upzeta^\alpha\upzeta^\beta>(\upzeta_\alpha u^\alpha)^2$, and $(\Proj^{\mu\nu}\upxi_\mu\upzeta_\nu)^2\le \Proj^{\mu\nu}\upxi_\mu\upxi_\nu \Proj^{\alpha\beta}\upzeta_\alpha\upzeta_\beta$. Thus, $\Lambda_{2,\pm}$ has two distinct roots giving two linearly independent eigenvectors.

\vskip 0.2cm
\underline{$G^3=0$} can also be written as $b^2 - \upbeta  a^\mu a_\mu = 0$, where 
$\upbeta = \frac{\upeta}{\uplambda}$. The roots are written the same way as $\Lambda_{2,\pm}$ with the particularity that now each one has multiplicity $3$. We note that these roots are real because $0 < \upbeta < 1$. The corresponding eigenvectors are
\begin{align}
\begin{bmatrix}
C_\pm\\ 
D_\pm^\nu\\ 
e_\pm^\nu \\ 
\frac{a_{0}^\pm e_\pm^\nu}{b_\pm}\\ 
\frac{a_{1}^\pm e_\pm^\nu}{b_\pm} \\
\frac{a_{2}^\pm e_\pm^\nu}{b_\pm} \\
\frac{a_{3}^\pm e_\pm^\nu}{b_\pm}\\ 
0_{5\times 1}
\end{bmatrix},
\nonumber
\end{align}
where $a_\pm$ is as $a$ above but with $\Lambda = \Lambda_{3,\pm}$,
$b_\pm$ is as $b$ above but with $\Lambda = \Lambda_{3,\pm}$ (so that
$b^2_\pm=\upbeta (a_\pm)^\mu(a_\pm)_\mu$),
\begin{align}
\begin{split}
C_\pm &=-\frac{\uplambda}{\uplambda+\upchi} ((2\uplambda+\upchi)
(e_\pm)^\mu (\Xi_\pm)_\mu -\frac{\uplambda}{3\upeta}(2\upeta+\upchi)
(a_\pm)^\mu (e_\pm)_\mu ),
\\
D^\mu_\pm &=\frac{\uplambda+\upchi}{3b_\pm^2\uplambda} (
(a_\pm)^\nu  (e_\pm)_\nu (a_\pm)^\mu -3b_\pm^2\upchi (e_\pm)^\mu-(e_\pm)^\nu D^{\mss\mu\lambda}_\nu (a^\pm)_\lambda ),
\end{split}
\nonumber
\end{align}
where $\Xi_\pm$ is as $\Xi$ above but with $\Lambda = \Lambda_{3,\pm}$,
and $e_\pm$ obeys the following constraint
\begin{align}
\frac{\uplambda+\upchi}{\uplambda\upchi}c_\pm b_\pm-\frac{3\upeta}{\uplambda} (\Xi_\pm)^\mu 
(e_\pm)_\mu +\frac{2\upeta-\uplambda}{\uplambda}(a_\pm)^\mu (e_\pm)_\mu=0.
\nonumber
\end{align}
Thus, the eigenvectors are written in terms of $3$ independent components of $e^\mu$ for each root, giving a total of $6$ eigenvectors.

\vskip 0.2cm
\underline{$FG-F\frac{\uplambda+\upeta}{\uplambda}a^\mu a_\mu -\upkappa = 0$}
can be written
as 
\begin{align}
9\uplambda \upchi b^4-6(\uplambda+2\upeta)\upchi a^\mu a_\mu b^2+\uplambda(\upchi-4\upeta)(a^\mu a_\mu)^2=0.
\nonumber
\end{align}
This is a quadratic equation for $b^2$ that has positive discriminant, i.e.,
\begin{align}
(a^\mu a_\mu)^2\upeta \upchi(\uplambda^2+\upeta\upchi+\uplambda\upchi)>0.
\nonumber
\end{align}
In order to obtain real roots $\Lambda$, we need
\begin{align}
0 < \frac{b^2}{a^\mu a_\mu }=\frac{2\upchi(\uplambda+\upeta)\pm\sqrt{\upeta \upchi(\uplambda^2+\upeta\upchi+\uplambda\upchi)}}{3\uplambda\upchi} \le 1.
\nonumber
\end{align}
This gives the condition
\begin{align}
2\upchi(\uplambda+\upeta)-\sqrt{\upeta \upchi(\uplambda^2+\upeta\upchi+\uplambda\upchi)} > 0,
\nonumber
\end{align}
which is satisfied in view of $\upchi > 4\uplambda$, and
\begin{align}
\frac{2\upchi(\uplambda+\upeta)+\sqrt{\upeta \upchi(\uplambda^2+\upeta\upchi+\uplambda\upchi)}}{3\uplambda\upchi}
\le 1,
\nonumber
\end{align}
which is satisfied in view of $\uplambda \ge \frac{3\upchi \uplambda}{\upchi - \uplambda}$.
We also observe that these four roots are distinct, so that we obtain four linearly independent 
eigenvectors.

Finally, we notice that condition $(i)$ can be verified upon setting $\upzeta=0$ in the above
computations.
\end{proof}

From the above Proposition, we immediately obtain:

\begin{corollary} \label{c:diagonalization}
Assume that 
$\upchi > 4 \upeta >0$ and that $\uplambda \ge \frac{3\upchi \upeta}{\upchi - \upeta}$.
Then, the system \eqref{E:Matrix_system_1st_order} can be written as
\begin{align}
\nabla_0 \mathbf{\Psi} + \tilde{\mathcal{A}}^i \nabla_i  \mathbf{\Psi}
 & = \tilde{\mathcal{R}},
 \label{E:First_order_system_unit_coefficient}
\end{align}
where $\tilde{\mathcal{A}}^i = (\mathcal{A}^0)^{-1} \mathcal{A}^i$
and $\tilde{\mathcal{R}} = -(\mathcal{A}^0)^{-1} \mathcal{R}$,
 and the eigenvalue
problem $(\tilde{\mathcal{A}}^i  \upzeta_i - \Lambda I )V = 0$ possesses
only real eigenvalues $\Lambda$ and a set of complete eigenvectors $V$.
\end{corollary}

\section{Energy estimates}\label{Section 4}

\subsection{Preliminaries}

We first set down some notations. 
Let $I=[0,T]$ for some $T>0$. 
We use $\mathscr{K}:\bR_+\to \bR_+$ to denote a continuous function which may vary from line to line. 
Similarly, $\mathscr{K}_I:\bR_+\to \bR_+$ denotes a continuous function depending on $I$.
Further, the notation $\fR$  always denotes a pseudodifferential operator ($\Psi$DO) whose mapping property may vary from line to line.  
We denote the $L^2$ based Sobolev space of order $r$ by $H^r$, with norm $\| \cdot \|_r$.

Due to the 
quasilinear nature of our equations, we will need to employ a pseudodifferential
calculus for symbols with limited smoothness. Such a calculus can be found 
in \cite{MarschallPSOSobolev,MarschallPSONonRegular,MarschallPSONonRegularCorrection},
to which we will refer frequently.
We denote the class of symbols on $\mathbb T^3$ of order $r$ with Sobolev regularity $k$ by $\cS^{r}_0(k,2)(\mathbb T^3)$. Given $a \in \cS^{r}_0(k,2)(\mathbb T^3)$, we denote the left quantization of $a$ by $Op(a)$ and the resulting space of $r$th order $\Psi$DO's by $OP \cS^r_0(k,2)$. For the reader's convenience, we recall the definition of these symbols and quantizations on $\mathbb R^3$ which then yield a $\Psi$DO calculus on any smooth closed manifold by the coordinate invariance of the definition and standard arguments (see \cite[Theorem 5.1, Corollary 5.2]{MarschallPSOSobolev}). 

\begin{definition}\cite{MarschallPSOSobolev}
Let $r\in \bR$ and $k>3/2$. Define $\cS^r_0(k,2)(\bR^3)=\cS^r_0(k,2)(\bR^3,\bC)$ to be the space of all symbols $a:\bR^3 \times \bR^3\to \bC$ such that for all spatial multi-indices $\vec{\alpha} = (\alpha_1,
\alpha_2,\alpha_3)$
\begin{align}
\begin{split}
|\partial^{\vec{\alpha}}_\upzeta a(x,\upzeta)| & \leq C_{\vec{\alpha}}(1+|\upzeta|)^{r-|\vec{\alpha}|},\\
\|\partial^{\vec{\alpha}}_\upzeta a(x,\upzeta)\|_{H^k} & \leq C_{\vec{\alpha}}(1+|\upzeta|)^{r -|\vec{\alpha}|}.
\end{split}
\nonumber
\end{align}
For a matrix-valued symbol $a:\bR^3 \times \bR^3\to \bC^{h\times l}$ with $h,l\in \mathbb{N}$, we say $a \in \cS^r_0(k,2)(\bR^3,\bC^{h \times l})$ if all the entries of $a$ belong to $\cS^r_0(k,2)(\bR^3)$.
The left quantization, $Op(a)$, of a symbol $a\in \cS^r_0(k,2)(\bR^3,\bC^{h\times l})$ is defined by
\begin{align}
Op(a)f(x):=\frac{1}{(2\pi)^n}\int_{\bR^3} e^{\ii x\cdot \upzeta} a(x,\upzeta)\hat{f}(\upzeta)\, d\upzeta
\nonumber
\end{align}
for $f\in \cS(\bR^3,\bC^l)$, the space of Schwartz functions in $\mathbb{R}^3$.
\end{definition}

Since we will be working exclusively on $\mathbb T^3$, we will simply write $\mathcal S^r_0(k,2)$ instead of $\cS^r_0(k,2)(\mathbb T^3)$, and we will not typically specify if the symbol is scalar or matrix valued since the context will be clear. The (flat) Laplacian on $\mathbb T^3$ is denoted by $\Delta$, and we define 
\begin{align}
\la \nabla \ra := (1 -\Delta)^{\frac 12},
\nonumber
\end{align} 
an element of $OP \cl S_0^1(k,2)$ for every $k \in \mathbb R$. Finally, we recall that 
\begin{align}
\| \cdot \|_{r} \simeq \| \la \nabla \ra^r \cdot \|_{0}. 
\nonumber
\end{align} 

\subsection{Main estimates}
We consider the linear system associated with 
\eqref{E:First_order_system_unit_coefficient}. Given $\tilde{v}$, we define the
operator $\mathcal{F}(\tilde{v})$ by
\begin{align}
\cF(\tilde{v})\tilde{u} = \partial_t \tilde{u} + \tilde{\cA}^i (\tilde{v}) \nabla_i \tilde{u},
\nonumber
\end{align}
where $\tilde{\cA}^i(\tilde{v})$ corresponds to the matrix
$\tilde{\cA}^i =(\cA^0)^{-1} \cA^i$ of Corollary \ref{c:diagonalization}, but with
the entries of the matrix computed using $\tilde{v}$.
Then the first order system~\eqref{E:Matrix_system_1st_order}, 
or, equivalently, the system \eqref{E:First_order_system_unit_coefficient}, can be written as
\begin{equation}
\label{1st_order_sys}
\cF(\tilde{u}) \tilde{u}   = \tilde{\cR}(\tilde{u}), \quad \tilde{u}(0)=\tilde{u}_0,
\end{equation}
where $\tilde{\cR}(\tilde{u})= -(\cA^0)^{-1}(r_1,\dots,r_6)^T$. Above and in
what follows, we make the following change of notation. 
We will use $\tilde{u}$ for a solution of \eqref{1st_order_sys} (and $\tilde{v}$ 
for the coefficients of the corresponding linear system) instead of $\mathbf{\Psi}$. 
This is because at this point we will think of a solution in abstract terms, i.e., 
as a map from a time interval to a suitable function space, and so we use a different
notation to highlight this point of view.

The goal of this section is to prove the following energy estimates. 
\begin{proposition}\label{p:energyest}
Let $r >  7/2$, $I \subset \bR$ and 
\begin{align}
\mathbb E_1(I) = C(I; H^r) \cap C^1(I; H^{r-1}).
\nonumber
\end{align}
There exist increasing functions $\tilde M, \omega: [0,\infty) \rightarrow (0,\infty)$ such that if $\tilde{u},\tilde{v} \in C^\infty(I \times \mathbb T^3)$ satisfy  
\begin{align}
\cF(\tilde{v}) \tilde{u} & = \tilde{\cR}(\tilde{v}),  \mbox{on } I \times \mathbb T^3,
\label{1st_order_sys_eq}
\end{align}
then for all $t \in I$, 
\begin{align}
\norm{\tilde{u}(t)}_r^2 \leq \tilde{M}(\| \tilde{v} \|_{L^\infty(I; H^{r-1})}) e^{ 
t\omega(\|\tilde{v} \|_{\bE_1(I)})} \left [\norm{\tilde{u}_0}_r^2 + \int_0^t \norm{\tilde{\cR}(\tilde{v}(s))}_r^2\, ds \right ] \label{Basic EE}
\end{align}
\end{proposition} 

\begin{proof}
For $\upzeta = \upzeta_i dx^i \in T^* \mathbb T^3$, let  $\tilde{\cA}=\tilde{\cA}(\tilde{v},\upzeta)=\tilde{\cA}^i(\tilde{v}) \upzeta_i$ and  $ \mathfrak{U} = Op(\tilde{\cA})$. 
From the discussion in Section~\ref{S:Diagonalization}, there exist a matrix $\cS=\cS(\tilde{v},\upzeta)$ and a diagonal matrix 
$\tilde{\mathcal{D}}=\tilde{\mathcal{D}}(\tilde{v},\upzeta)$ such that
\begin{align}
\cS \tilde{\mathcal{A}} =  \tilde{\mathcal{D}} \cS.
\nonumber
\end{align}
We put  $\mathfrak{S}:=Op(\cS)$ and $\tilde{\mathfrak{D}}:=Op(\tilde{\mathcal{D}})$.
Based on the expression of $\tilde{\cA}^i(\tilde{v})\upzeta_i$, it is not hard to see that all its entries belong to $\cS^1_0(r,2)$. Denote by $\Lambda_k=\Lambda_k(\tilde{v},\upzeta)$ all  the (distinct) eigenvalues of $\tilde{\mathcal{A}}$.
Note that $\partial_\upzeta^{\vec{\alpha}} \tilde{\cA}(\tilde{v},\upzeta)$ is homogeneous of degree $1-
|\vec{\alpha}| $ for $|\vec{\alpha}|\leq 1$ and $\partial_\upzeta^{\vec{\alpha}} \tilde{\cA}(\tilde{v},\upzeta)=0$ for $|\vec{\alpha}|> 1$. 
We thus infer that $\Lambda_k/|\upzeta|$ is homogeneous in $\upzeta$ of degree zero.

Because the map $[(\tilde{v},\upzeta) \mapsto \Lambda_k(\tilde{v},\upzeta)]\in C^\infty(H^r \times T^* \mathbb T^3 , H^r)$,  
it follows that 
\begin{align}
\| \Lambda_k(\tilde{v},\upzeta)\|_r   \leq C,\quad |\upzeta|=1
\nonumber
\end{align}
for some $C=C(\norm{\tilde{v}}_r)$.
By the homogeneity of $\Lambda_k/|\upzeta|$, we can derive that
\begin{align}
\| \Lambda_k(\tilde{v},\upzeta)\|_r   \leq C (1+|\upzeta|),
\nonumber
\end{align}
for all $\upzeta$ and some $C=C(\norm{\tilde{v}}_r)$. 
Differentiating the characteristic polynomial of $\tilde{\mathcal{A}}$ with respect to $\upzeta$ and using induction immediately yield
\begin{align}
\label{Symbol est 2}
\|\partial^{\vec{\alpha}}_\upzeta \Lambda_k(\tilde{v},\upzeta)\|_r   \leq C_{\vec{\alpha}} (1+|\upzeta|)^{1 -
|\vec{\alpha|}},
\end{align}
for all $\upzeta$ and some $C_{\vec{\alpha}}=C_{\vec{\alpha}}(\norm{\tilde{v}}_r)$. 
By Sobolev embedding, this implies that $\Lambda_k\in \cS^1_0(r,2) $ and thus
\begin{align}
\tilde{\mathfrak{D}} \in OP\cS^1_0(r,2).
\nonumber
\end{align}
The projection onto the eigenspace associated to the eigenvalue $\Lambda_k$ is given by
\begin{align}
\label{Projection}
P_k=P_k(\tilde{v},\upzeta)= \frac{1}{2\pi \ii}\int_{\gamma_k} (z - \tilde{\cA}(\tilde{v},\upzeta))^{-1}\, dz,
\end{align}
where $\gamma_k$ is a smooth contour enclosing only one pole $\Lambda_k$. 
Note that with properly chosen contours $\gamma_k$, we can always make the eigenvalues $\tilde{\Lambda}_i(z,\tilde{v},\upzeta)$ of $(z-\tilde{\cA}(\tilde{v},\upzeta))^{-1}$ satisfy   
\begin{align}
 \| \tilde{\Lambda}_i(z,\tilde{v},\upzeta) \|_r \leq C=C(\|\tilde{v} \|_r),\quad |\upzeta|\leq 1, \, z\in \gamma_k 
 \nonumber
\end{align}
for all $k$.
From the homogeneity of $\tilde{\cA}$ and $\Lambda_k$, we infer that $P_k$ is homogeneous of degree $0$ in $\upzeta$. 
Combining with  \eqref{Symbol est 2} and \eqref{Projection}, we can derive that 
\begin{align}
  \| P_k(\tilde{v},\upzeta) \|_{H^r} & \leq C=C(\|\tilde{v}\|_r) ,\quad |\upzeta|=1.
  \nonumber
\end{align}
In view of the homogeneity of $P_k(\tilde{v},\cdot)$, this implies for all $\upzeta$ 
\begin{align}
  \| P_k(\tilde{v},\upzeta) \|_{H^r} & \leq C=C(\|\tilde{v} \|_r).
  \nonumber
\end{align}
Note that, for a given pair of $(\tilde{v},\upzeta)$, we can choose the contour $\gamma_k$ in \eqref{Projection} to be fixed in a neighborhood of  $(\tilde{v},\upzeta)$.
Applying a similar argument to the $\upzeta$-derivatives of $P_k$ and using the homogeneity 
of $\partial_\upzeta^{\vec{\alpha}} \tilde{\cA} $, 
direct computations lead to
$P_k\in \cS^0_0(r,2)$. This implies that 
\begin{align}\label{sym est of S}
\mathcal{S}=\mathcal{S}(\tilde{v},\upzeta) \in \cS^0_0(r,2)
\end{align}
and thus
\begin{align}
\mathfrak{S}=\mathfrak{S}(\tilde{v}) \in OP\cS^0_0(r,2)
\nonumber
\end{align}
with norm  depending on $\norm{\tilde{v}}_r$.

Then it follows from  \cite[Corollary~3.4]{MarschallPSOSobolev} that
\begin{align}
\mathfrak{S} \mathfrak{U} = \tilde{\mathfrak{D} } \mathfrak{S} + \fR
\nonumber
\end{align}
with
\begin{align}
\fR \in \cL(H^s, H^s) ,\quad 1-r<s\leq r-2.
\nonumber
\end{align}
We write $\mathfrak{U}=\ii \mathfrak{A} \la \nabla \ra$.
Let $\cA=\cA(\upzeta) $ denote the symbol of $\mathfrak{A}$, i.e. $\cA=-\ii \tilde{\cA}/(1+|\upzeta|^2)^{\frac 12}$. 
Hence $\mathfrak{A} \in OP\cS^0_0(r,2)$.
Then there exists a $\Psi$DO $\mathfrak{D} $ with symbol $\mathcal{D} \in \cS^0_0(r,2)$ such that
\begin{align}
\cS  \mathcal{A} =   \mathcal{D}  \cS  
\nonumber 
\end{align}
and thus 
\begin{align}
\mathfrak{S} \mathfrak{A} =  \mathfrak{D}  \mathfrak{S} + \fR 
\nonumber
\end{align}
with 
\begin{align}
\label{Commutator 3}
\fR \in \cL(H^{s-1}, H^s) ,\quad 1-r<s\leq r-1.
\end{align}
We rewrite \eqref{1st_order_sys_eq} as 
\begin{align}
\partial_t \tilde{u} & =  \ii \mathfrak{A}(\tilde{v}) \la \nabla \ra \tilde{u} 
+\tilde{\cR}(\tilde{v}),
\nonumber
\end{align}
or 
\begin{align}
\partial_t \tilde{u} & = \mathfrak{U}(\tilde{v}) \tilde{u} +\tilde{\cR}(\tilde{v}).
\nonumber
\end{align}

Denote by $\cS^*$ the conjugate transpose matrix of $\cS$. We further set $\tilde{\mathfrak{S}}:=Op(\cS^*)$.
Note that  $  \tilde{\mathfrak{S}}=\tilde{\mathfrak{S}}(v)\in OP\cS^0_0(r,2)$.  
Since $\cS$ is homogeneous of degree $0$ in $\upzeta$, combining with the discussion in Section~\ref{S:Diagonalization}, we infer  that
\begin{align}
\tilde{u}^T \cS^*(v,\upzeta)  \cS(v,\upzeta) \tilde{u} \geq C_0 |\tilde{u}|^2
\nonumber
\end{align}
for some $C_0=C_0(\|\tilde{v}\|_\infty) >0$.
Let $\cB=\cB(\tilde{v},\upzeta)=\sqrt{\cS^*(\tilde{v},\upzeta)  \cS(\tilde{v},\upzeta) -\frac{C_0}{2}I}$ and $\mathfrak{B}=Op(\cB)$. 
Here $I$ is the identity matrix and,  for a positive definite  matrix $A$, $B=\sqrt{A}$ denotes the square-root matrix of $A$, i.e. $B^* B=A$.
It is not hard to conclude that $\cB\in \cS^0_0(r,2)$ via the Cholesky algorithm. 
Putting $\tilde{\mathfrak{B}}=Op(\cB^*)\in OP\cS^0_0(r,2)$, it follows from \cite[Corollaries~3.4 and 3.6]{MarschallPSOSobolev} that
\begin{align}
\fR=&\tilde{\mathfrak{S}}\circ \mathfrak{S} -\frac{C_0}{2}I - \mathfrak{B}^* \mathfrak{B} \\
\label{S4: commutator 0}
=&[(\tilde{\mathfrak{S}}\circ \mathfrak{S} -\frac{C_0}{2}I )-  \tilde{\mathfrak{B}} \circ \mathfrak{B}]  
+(\tilde{\mathfrak{B}} \circ \mathfrak{B}  - \tilde{\mathfrak{B}}   \mathfrak{B} )  
+( \tilde{\mathfrak{B}}   \mathfrak{B} - \mathfrak{B}^* \mathfrak{B} )  \in \cL(H^{s-1},H^s)  
\end{align}
for all $1-r<s< r.$
Define
\begin{align}
N_r(t) := \la \nabla \ra^r (\frac{C_0}{2}I  + \mathfrak{B}^* \mathfrak{B})\la \nabla \ra^r .
\nonumber
\end{align}
It is an immediately conclusion from its definition that
\begin{align}
\label{S4: Garding type}
(N_r(t) \tilde{u} , \tilde{u}) \geq \frac{C_0}{2} \| \tilde{u} \|_r^2.
\end{align}
We have 
\begin{align}
\begin{split}
N_r 
=& \la \nabla \ra^r (\frac{C_0}{2}I  + \mathfrak{B}^* \mathfrak{B} - \tilde{\mathfrak{S}}\circ \mathfrak{S} )\la \nabla \ra^r  
+ \la \nabla \ra^r (\tilde{\mathfrak{S}}\circ \mathfrak{S}-\tilde{\mathfrak{S}}  \mathfrak{S}) \la \nabla \ra^r\\ 
& +  \la \nabla \ra^r  ( \tilde{\mathfrak{S}}   - \mathfrak{S}^*      ) \mathfrak{S}\la \nabla \ra^r  + \la \nabla \ra^r  \mathfrak{S}^*  \mathfrak{S} \la \nabla \ra^r.
\end{split}
\nonumber
\end{align}
It follows from  \cite[Corollary 3.4]{MarschallPSOSobolev} that 
\begin{align}\label{S4: commutator 1}
\tilde{\mathfrak{S}}\circ \mathfrak{S}-\tilde{\mathfrak{S}}  \mathfrak{S} \in \cL(H^{s-1},H^s), \quad 1-r <s\leq r,
\end{align}
and from \cite[Corollary 3.6]{MarschallPSOSobolev}    that 
 \begin{align}\label{S4: commutator 2}
\tilde{\mathfrak{S}}   - \mathfrak{S}^* \in \cL(H^{s-1},H^s), \quad 1-r <s< r,
\end{align}

We compute
\begin{align}
\begin{split}
\frac{d}{dt} (N_r \tilde{u},  \tilde{u}) 
& = 
(N_r \frac{d}{dt} \tilde{u},  \tilde{u})
+
(N_r \tilde{u} , \frac{d}{dt} \tilde{u})
+ 
(N_r^\prime  \tilde{u},   \tilde{u})
\\
& 
= 
(N_r \mathfrak{U} \tilde{u},  \tilde{u})
+ 
(N_r  \tilde{\cR},  \tilde{u})
+
(N_r \tilde{u}, \mathfrak{U} \tilde{u})
+
(N_r \tilde{u}, \tilde{\cR})
+ 
(N_r^\prime  \tilde{u},   \tilde{u})
\\
&= 
( (N_r \mathfrak{U} + \mathfrak{U}^* N_r )\tilde{u},  \tilde{u})
+ 
(N_r  \tilde{\cR},  \tilde{u})
+
(N_r \tilde{u}, \tilde{\cR})
+ 
(N_r^\prime  \tilde{u},   \tilde{u}),
\end{split}
\nonumber
\end{align}
where ${}^\prime=\frac{d}{dt}$. 
We have
\begin{align}
\begin{split}
N_r \mathfrak{U} + \mathfrak{U}^* N_r  =& 
[\la \nabla \ra ^r (\frac{C_0}{2}I  + \mathfrak{B}^* \mathfrak{B}) \la \nabla \ra ^r  ] \ii \mathfrak{A}\la \nabla \ra  \\
& -\ii \la \nabla \ra  \mathfrak{A}^*  
[\la \nabla \ra ^r (\frac{C_0}{2}I  + \mathfrak{B}^* \mathfrak{B}) \la \nabla \ra ^r  ]
\\
 = &
\ii [\la \nabla \ra ^r (\frac{C_0}{2}I  + \mathfrak{B}^* \mathfrak{B}) \la \nabla \ra ^r\mathfrak{A}\la \nabla \ra \\
&-\la \nabla \ra  \mathfrak{A}^* \la \nabla \ra ^r (\frac{C_0}{2}I  + \mathfrak{B}^* \mathfrak{B}) \la \nabla \ra ^r ].
\end{split}
\nonumber
\end{align}

Note that $\la \nabla \ra ^r \in OP\cS^r_0(k,2)$ for any $k$. 
We can infer from \eqref{S4: commutator 0}, \eqref{S4: commutator 1} and \eqref{S4: commutator 2} that
\begin{align}
\begin{split}
N_r \mathfrak{U} =& \ii N_r \mathfrak{A}\la \nabla \ra \\
=& \ii \la \nabla \ra ^r \mathfrak{S}^* \mathfrak{S} \la \nabla \ra ^r\mathfrak{A}\la \nabla \ra  +\fR ,
\end{split}
\nonumber
\end{align}
where $\fR=\fR(v) \in \cL(H^r, H^{-r})$.

To estimate the first term in the second line, we first notice that   \cite[Corollary 3.4]{MarschallPSOSobolev} implies
\begin{align}
\mathfrak{b}=\mathfrak{b}(\tilde{v}):=[\la \nabla \ra ^r,\mathfrak{A}(v) ] \in \cL(H^{r-1}, H^0),
\nonumber
\end{align}
and again, its norm depends on  $\|\tilde{v}\|_r$. 
Thus we have that 
\begin{align}
\la \nabla \ra ^r \mathfrak{S}^* \mathfrak{S} \la \nabla \ra ^r\mathfrak{A}\la \nabla \ra  & =
\la \nabla \ra ^r \mathfrak{S}^* \mathfrak{S}\mathfrak{A}\la \nabla \ra   \la \nabla \ra ^r +
\fR,
\nonumber
\end{align}
where $\fR=\fR(\tilde{v}) \in \cL(H^r, H^{-r})$.
Now observe that by \eqref{Commutator 3}
\begin{align}
\begin{split}
\mathfrak{S}\mathfrak{A}\la \nabla \ra  & =   \mathfrak{D} \mathfrak{S}\la \nabla \ra  + \fR 
\\
& = \mathfrak{D}\la \nabla \ra  \mathfrak{S} + \fR,
\end{split}
\nonumber
\end{align}
where in   the second equality we used \cite[Corollary 3.4]{MarschallPSOSobolev}. 
Here and below the operator $\fR$ may vary from line to line, but all these $\fR$ satisfy
\begin{align}
\fR=\fR(\tilde{v}) \in \cL(H^s, H^s) \quad \text{for all }-r+1<s\leq r-1.
\nonumber
\end{align} 
Therefore,
\begin{align}
N_r \mathfrak{U}  & =
\ii\la \nabla \ra ^r \mathfrak{S}^* \mathfrak{D}\la \nabla \ra  \mathfrak{S}   \la \nabla \ra ^r +
\fR,
\nonumber
\end{align}
where $\fR=\fR(v) \in \cL(H^r, H^{-r})$ and its norm depends on  $\|\tilde{v}\|_r$.

We can carry out a similar analysis for the term $\mathfrak{U}^* N_r$. More precisely, first notice that
\begin{align}
\begin{split}
\mathfrak{U}^* N_r = & -\ii \la \nabla \ra  \mathfrak{A}^*    \la \nabla \ra ^r  \mathfrak{S}^*  \mathfrak{S} \la \nabla \ra ^r \\
&-\ii \la \nabla \ra  \mathfrak{A}^*[ \la \nabla \ra ^r (\frac{C_0}{2}I  + \mathfrak{B}^* \mathfrak{B} - \tilde{\mathfrak{S}}\circ \mathfrak{S} )\la \nabla \ra ^r \\
&  +  \la \nabla \ra ^r (\tilde{\mathfrak{S}}\circ \mathfrak{S}-\tilde{\mathfrak{S}}  \mathfrak{S}) \la \nabla \ra ^r ] \\
&  +\la \nabla \ra ^r ( \tilde{\mathfrak{S}}   - \mathfrak{S}^*      ) \mathfrak{S}\la \nabla \ra ^r.
\end{split}
\nonumber
\end{align}
Using \eqref{S4: commutator 0}, \eqref{S4: commutator 1}, \eqref{S4: commutator 2} and \cite[Theorem 2.4]{MarschallPSOSobolev}, we infer that the last three terms on the right-hand side belong to $ \cL(H^r, H^{-r})$.
As 
\begin{align}
-\ii \la \nabla \ra  \mathfrak{A}^*    \la \nabla \ra ^r  \mathfrak{S}^*  \mathfrak{S} \la \nabla \ra ^r= [\ii     \la \nabla \ra ^r  \mathfrak{S}^*  \mathfrak{S} \la \nabla \ra ^r \mathfrak{A} \la \nabla \ra   ]^*,
\nonumber
\end{align}
we  conclude that
\begin{align}
\begin{split}
N_r \mathfrak{U} + \mathfrak{U}^* N_r  =& 
\ii [\la \nabla \ra ^r \mathfrak{S}^* \mathfrak{D}\la \nabla \ra  \mathfrak{S}   \la \nabla \ra ^r\\
&-\la \nabla \ra ^r \mathfrak{S}^* \la \nabla \ra  \mathfrak{D}^* \mathfrak{S}\la \nabla \ra ^r ] + \fR_0,
\\
 = &
\ii \la \nabla \ra ^r \mathfrak{S}^*[  \mathfrak{D}\la \nabla \ra    
-  \la \nabla \ra  \mathfrak{D}^*  ]\mathfrak{S} \la \nabla \ra ^r + \fR_0,
\end{split}
\nonumber
\end{align}
where $\fR_0=\fR_0(\tilde{v}) \in \cL(H^r, H^{-r})$ with norm depending on  $\|\tilde{v}\|_r$.
The term in the parenthesis is bounded in $\cL(H^0)$ due to \cite[Corollary 3.6]{MarschallPSOSobolev}.  
We thus have
\begin{align}
\frac{d}{dt} (N_r \tilde{u},  \tilde{u}) 
 = &
\ii
(   \la \nabla \ra ^r [ \mathfrak{S}^* \mathfrak{D}\la \nabla \ra  \mathfrak{S}   
-\mathfrak{S}^* \la \nabla \ra  \mathfrak{D}^* \mathfrak{S} ] \la \nabla \ra ^r \tilde{u},  
\tilde{u})\\
& + (\fR_0 \tilde{u},  \tilde{u})
+ 
(N_r  \tilde{\cR},  \tilde{u})
+
(N_r \tilde{u}, \tilde{\cR})
+ 
(N_r^\prime  \tilde{u},   \tilde{u}). \label{eq:energy_der}
\end{align}
We have
\begin{align}
| \ii
(   \la \nabla \ra ^r [ \mathfrak{S}^* \mathfrak{D}\la \nabla \ra  \mathfrak{S}   
-\mathfrak{S}^* \la \nabla \ra  \mathfrak{D}^* \mathfrak{S} ] \la \nabla \ra ^r \tilde{u},  \tilde{u})| 
& 
\leq C_1 \norm{\tilde{u}}_r^2,
\nonumber
\end{align}
\begin{align}
|(\fR_0  \tilde{u}, \tilde{u})|
 \leq C_2 \norm{\fR_0 \tilde{u}}_{-r}
 \norm{ \tilde{u} }_{r}
  \leq C_2
  \norm{ \tilde{u} }_{r}^2,
  \nonumber
\end{align}
\begin{align}
|(N_r  \tilde{\cR},  \tilde{u})|
+
|(N_r \tilde{u}, \tilde{\cR})|
& \leq C_3 \norm{\tilde{u}}_r \norm{\tilde{\cR}}_r \leq C_4 \norm{\tilde{u}}_r^2 + \frac{1}{2} \norm{\tilde{\cR}}_r^2.
\nonumber
\end{align}
Here the constants $C_i$ all depend on $\|\tilde{v}\|_r$.
To estimate the last term in \eqref{eq:energy_der}, observe that
\begin{align}
N^\prime(t) = \la \nabla \ra ^r \partial [\mathfrak{B}^*(\tilde{v})  \mathfrak{B}(\tilde{v})]  
\tilde{v}^\prime \la \nabla \ra ^r .
\nonumber
\end{align}
Here $\partial$ stands for the Frech\'et derivative. 
From \eqref{Projection} and \eqref{sym est of S}, it is not hard to see that
\begin{align}
\partial \mathcal{B} (\tilde{v})\tilde{v}^\prime  \in \cS^0_0(r-1,2).
\nonumber
\end{align}
Hence \cite[Theorem~2.3]{MarschallPSOSobolev}  implies that
\begin{align}
\partial \mathfrak{B}(\tilde{v}) \tilde{v}^\prime = Op(\partial \mathcal{B} (\tilde{v})\tilde{v}^\prime  ) \in \cL(H^0).
\nonumber
\end{align}
As $\partial \mathfrak{B}^*(\tilde{v}) \tilde{v}^\prime =[\partial \mathfrak{B}(\tilde{v}) \tilde{v}^\prime]^*$, 
we immediate conclude that
\begin{align}
\partial [\mathfrak{B}^*(\tilde{v})  \mathfrak{B}(\tilde{v})]  \tilde{v}^\prime \in \cL(H^0).
\nonumber
\end{align}
Now it follows that
\begin{align}
|(N_r^\prime  \tilde{u},   \tilde{u})|
& \leq C_6  \norm{N_r^\prime  \tilde{u}}_{-r} \norm{\tilde{u}}_r  \leq C_6 
\norm{ \tilde{u}}_{r}^2,
\nonumber
\end{align}
where $C_6$   depends on $\|\tilde{v}\|_{\bE_1(I)}$. In summary, 
\begin{align}
\frac{d}{dt} (N_r \tilde{u},  \tilde{u})  \leq C_7 \norm{\tilde{u}}_r^2 + C_5\norm{\tilde{\cR}}_r^2
\label{EE inequality}
\end{align}
with $C_7=C_7(\|\tilde{v}\|_{\bE_1(I)})$.
As a direct conclusion from \eqref{S4: Garding type} and Gr\"onwall's inequality, we have
\begin{align}
\norm{\tilde{u}(t)}_r^2 \leq \tilde{M} e^{ t\omega(\|\tilde{v}\|_{\bE_1(I)})} \left [\norm{\tilde{u}_0}_r^2 + \int_0^t \norm{\tilde{\cR}(\tilde{v}(s))}_r^2\, ds \right ]
\nonumber
\end{align}
where $\tilde{M}$ is a constant argument depending on $\|\tilde{v}\|_\infty$, and thus, 
on $\|\tilde{v}\|_{r-1}$ by Sobolev embedding.
\end{proof}


\section{Local existence and uniqueness}
In this section, we use the energy estimate of Proposition \ref{p:energyest} to establish
local well-posedness for the system \eqref{E:Matrix_system_1st_order},
which in turn will imply Theorem \ref{T:main_theorem}.

\subsection{Approximating sequence}\label{Section 5.1}
We take a sequence of smooth initial data $\tilde{u}_{0,n} \to \tilde{u}_0$ in $H^r$ with $r> 7/2$.
Then we inductively study
\begin{align}
\cF(\tilde{u}_{n-1}) \tilde{u}_n   &= \tilde{\cR}(\tilde{u}_{n-1}), \quad 
\tilde{u}_n(0)=\tilde{u}_{0,n}.
\label{1st_order_sys_ind}
\end{align}
Let $\norm{\tilde{u}_0}_r^2 \leq K$. We may assume
\begin{align}
\label{asp on approx initial}
\norm{\tilde{u}_{0,n}}_r^2 \leq K+1.
\end{align}
Further, we define continuous functions $\mathscr{K}_i:\bR_+ \to \bR_+$ with $i=1,2$ such that
\begin{align}
\nonumber
\| \mathfrak{A}(\tilde{v}) \|_{\cL(H^{r-1})} \leq \mathscr{K}_1(\| \tilde{v} \|_r)
\end{align}
and
\begin{align}
\nonumber
\| \tilde{\cR}(\tilde{v}) \|_{s} \leq \mathscr{K}_2(\| \tilde{v} \|_s),\quad s=r-1,r.
\end{align}
We next make the inductive assumption
\begin{align}
\nonumber
H(n-1):  \norm{\tilde{u}_k}_{C(I;H^r)} \leq \cC_1 \, \text{ and } \, 
\|\partial_t \tilde{u}_k\|_{C(I;H^{r-1})}\leq \cC_2 \, \text{ for } \, k=1,2,\cdots, n-1.
\end{align}
Note that it follows from $H(n-1)$ and \eqref{asp on approx initial} that by choosing $T$ small enough, we have
\begin{align}
\nonumber
\norm{\tilde{v}_k(t)}_{r-1} \leq M,\quad  k=1,2,\cdots, n-1 \text{ and } t\in [0,T]
\end{align}
for some sufficiently large uniform constant $M$ independent of $\cC_i$.
As a direct consequence, we can take the constant $\tilde{M}$ in \eqref{Basic EE} to be uniform in the following iteration argument.

Furthermore, we choose $\cC_i$ in $H(n-1)$ large enough so that
\begin{align}
\sqrt{\tilde{M} (2K+4)} \leq \cC_1
\nonumber
\end{align}
and
\begin{align}
M^\prime \mathscr{K}_1(\cC_1)\cC_1 + \mathscr{K}_2(\cC_1)\leq \cC_2,
\nonumber
\end{align}
where $M^\prime=\| \la \nabla \ra \|_{\cL(H^r,H^{r-1})}$.
Now we will use \eqref{Basic EE} to estimate 
\begin{align}
\norm{\tilde{u}_n(t)}_r^2 \leq  &     \tilde{M} e^{ t\omega( \norm{\tilde{u}_{n-1}}_{\bE_1(I)})} [\norm{\tilde{u}_{0,n}}_r^2 + \int_0^t \norm{\tilde{\cR}(\tilde{u}_{n-1}(s))}_r^2\, ds] \\
\leq & \tilde{M} e^{ t\omega( \cC_1+\cC_2)} [K+1 + t \mathscr{K}_2(\cC_1)],
\nonumber
\end{align}
By choosing $T$ small enough, we can control
\begin{align}
\nonumber
\norm{\tilde{u}_n(t)}_r^2  \leq \tilde{M}(2K+4) \quad \text{for all }t\in [0,T],
\end{align}
which gives
\begin{align}
\nonumber
\norm{\tilde{u}_n}_{C(I;H^r)} \leq \cC_1.
\end{align}
We plug this estimate into \eqref{1st_order_sys_ind} and thus obtain
\begin{align}
\| \partial_t \tilde{u}_n(t)\|_{r-1} \leq M^\prime \mathscr{K}_1(\cC_1)\cC_1 + \mathscr{K}_2(\cC_1)\leq \cC_2.
\nonumber
\end{align}
This completes the verification of $H(n)$. One thus infers that
\begin{align}
\norm{\tilde{u}_n}_{\bE_1(I)} \leq \cC
\label{Bound EE}
\end{align}
for all $n$ and some $\cC>0$.

\subsection{Energy estimate for the difference of two solutions}
For $i=1,2$, we consider
\begin{align}
\nonumber
\cF(\tilde{v}_i) \tilde{w}_i   &= \tilde{\cR}(\tilde{v}_i), \quad \tilde{w}_i(0)=\tilde{w}_{0,i}.
\end{align}
Set $\tilde{v}=\tilde{v}_2-\tilde{v}_1$ and $\tilde{w}=\tilde{w}_2-\tilde{w}_1$. 
Taking the difference of the above two systems, we obtain
\begin{align}
\partial_t \tilde{w} &= \mathfrak{U}(\tilde{v}_2) \tilde{w}
 + [\mathfrak{U}(\tilde{v}_1)-\mathfrak{U}(\tilde{v}_2)] \tilde{w}_1 +\tilde{\cR}(\tilde{v}_2)-\tilde{\cR}(\tilde{v}_1)
 , \quad \tilde{w}(0)=\tilde{w}_{0,2}-\tilde{w}_{0,1}.
\label{1st_order_sys_diff}
\end{align}
Let
\begin{align}
\nonumber
\mathfrak{F}=[\mathfrak{U}(\tilde{v}_1)-\mathfrak{U}(\tilde{v}_2)] 
\tilde{w}_1 +\tilde{\cR}(\tilde{v}_2)-\tilde{\cR}(\tilde{v}_1)
\end{align}
and
\begin{align}
\nonumber
\bE_0(I):= C(I; H^{r-1})\cap C^1(I; H^{r-2}).
\end{align}
By \eqref{Basic EE}, we have
\begin{align}
\nonumber
\norm{\tilde{w}(t)}_{r-1}^2 \leq \tilde{M} e^{ t\omega} 
[\norm{\tilde{w}_{0,2}-\tilde{w}_{0,1}}_{r-1}^2 + \int_0^t \norm{\mathfrak{F}(s)}_{r-1}^2\, ds].
\end{align}
Here $\tilde{M}=\tilde{M}(\norm{\tilde{v}_2}_{r-2})$ and $\omega=\omega(\norm{\tilde{v}_2}_{\bE_0(I)})$.
We estimate
\begin{align}
\begin{split}
 \|[\mathfrak{U}(\tilde{v}_1)-\mathfrak{U}(\tilde{v}_2)] \tilde{w}_1\|_{r-1}  
\leq & \int_0^1 \| \partial \mathfrak{U} (s \tilde{v}_1 + (1-s)\tilde{v}_2)(\tilde{v}) 
\tilde{w}_1\|_{r-1} ds \\
\leq & \int_0^1 \| \partial \mathfrak{A} (s \tilde{v}_1 + (1-s)\tilde{v}_2)(\tilde{v}) 
\la \nabla \ra  \tilde{w}_1\|_{r-1} ds \\
\leq & \mathscr{K}(\norm{\tilde{v}_1}_{r-1} + \norm{\tilde{v}_2}_{r-1}) \norm{\tilde{v}}_{r-1} \|\tilde{w}_1\|_r.
\end{split}
\nonumber
\end{align}
Similarly,
\begin{align}
\nonumber
\|\tilde{\cR}(\tilde{v}_2)-\tilde{\cR}(\tilde{v}_1)\|_{r-1} \leq \mathscr{K}(\norm{\tilde{v}_1}_{r-1}+ \norm{\tilde{v}_2}_{r-1}) \norm{\tilde{v}}_{r-1}.
\end{align}
This yields
\begin{align}
\norm{\tilde{w}(t)}_{r-1}^2  
\leq  & \tilde{M} e^{ t\omega } [\norm{\tilde{w}_{0,2}-\tilde{w}_{0,1}}_{r-1}^2 \\
& + t (1+\|\tilde{w}_1\|_r^2) \mathscr{K}(\norm{\tilde{v}_1}_{r-1} + \norm{\tilde{v}_2}_{r-1}) \norm{\tilde{v}}_{r-1}^2 ] .
\label{EE_diff_1}
\end{align}
Using \eqref{1st_order_sys_diff}, we further have
\begin{align}
\|\partial_t \tilde{w}\|_{r-2}  \leq \mathscr{K}(\norm{\tilde{v}_2}_{r-2})\norm{\tilde{w}}_{r-1} 
 +  \|\mathfrak{F}\|_{r-2} .
\label{EE_diff_2}
\end{align}

\subsection{Convergence}
Now we choose $\tilde{v}_2=\tilde{w}_1=\tilde{u}_{n-1}$, $\tilde{v}_1=\tilde{u}_{n-2}$ and $\tilde{w}_2=\tilde{u}_n$. 
Note that as in Section~\ref{Section 5.1}, the constant $\tilde{M}$ in \eqref{EE_diff_1} can be taken to be independent of $n$.
\eqref{EE_diff_1} and \eqref{EE_diff_2} show that
\begin{align}
\begin{split}
&\|\tilde{u}_n-\tilde{u}_{n-1}\|_{\bE_0(I)}\\
  \leq & \sqrt{\tilde{M}} e^{\frac{T}{2} \omega(\cC)}[\| \tilde{u}_{0,n}-\tilde{u}_{0,n-1}\|_{r-1}+ \sqrt{T}(1+\cC)\mathscr{K}_I(\cC)  \|\tilde{u}_{n-1}-\tilde{u}_{n-2}\|_{\bE_0(I)} ]\\
& + \mathscr{K}_I(\cC)\sqrt{\tilde{M}} e^{\frac{T}{2} \omega(\cC)}[\| \tilde{u}_{0,n}-\tilde{u}_{0,n-1}\|_{r-1}+ \sqrt{T}(1+\cC)\mathscr{K}_I(\cC)  \| \tilde{u}_{n-1} - \tilde{u}_{n-2} \|_{\bE_0(I)} ]\\
&+ \mathscr{K}_I(\cC) \sup\limits_{t\in I} \|\tilde{u}_{n-1}(t)-\tilde{u}_{n-2}(t)\|_{r-2}.
\end{split}
\nonumber
\end{align}
In the last line, we can use \eqref{EE_diff_1}  once more to obtain
\begin{align}
\begin{split}
&\sup\limits_{t\in I} \|\tilde{u}_{n-1}(t)-\tilde{u}_{n-2}(t)\|_{r-1} \\
\leq & \sqrt{\tilde{M}} e^{\frac{T}{2} \omega(\cC)}[\|\tilde{u}_{0,n-1}-\tilde{u}_{0,n-2}\|_{r-1}+ \sqrt{T}(1+\cC)\mathscr{K}_I(\cC)  \|\tilde{u}_{n-2}-\tilde{u}_{n-3}\|_{\bE_0(I)} ]
\end{split}
\nonumber
\end{align}
We can choose $T$ small and $(\tilde{u}_{0,n})$ in such a way that
\begin{align}
\nonumber
(1+\mathscr{K}_I(\cC))\sqrt{\tilde{M}} e^{\frac{T}{2} \omega(\cC)}(\|\tilde{u}_{n-1}(t)-\tilde{u}_{n-2}(t)\|_{r-1} + \| \tilde{u}_{0,n-1}-\tilde{u}_{0,n-2}\|_{r-1} ) \leq 2^{-n},
\end{align}
\begin{align}
\nonumber
\sqrt{\tilde{M}} e^{\frac{T}{2} \omega(\cC)}\sqrt{T}(1+\cC)\mathscr{K}_I^2(\cC)  \leq 1/16,
\end{align}
and
\begin{align}
\nonumber
\sqrt{\tilde{M}} e^{\frac{T}{2} \omega(\cC)}\sqrt{T}(1+\cC)(\mathscr{K}_I^2(\cC) +\mathscr{K}_I(\cC) )  \leq 1/4.
\end{align}
Putting $a_n=\|\tilde{u}_n-\tilde{u}_{n-1}\|_{\bE_0(I)}$. We thus infer that
\begin{align}
\nonumber
a_n \leq 2^{-n} + a_{n-1}/4 + a_{n-2}/16.
\end{align}
By induction, one can show that
\begin{align}
a_n \leq \frac{s_n}{2^{2n-3}} + \frac{F_n}{2^{2n-4}}a_2 + \frac{F_{n-1}}{2^{2n-2}} a_1,
\label{Inductive: a_n}
\end{align} 
where $F_n$ is the $n-$th term of Fibonacci sequence (starting from $0$) and
\begin{align}
\nonumber
s_n=2^{n-3} + s_{n-1}+s_{n-2}.
\end{align}
Let $b=(1-\sqrt{5})/2$. Then
\begin{align}
s_n - b s_{n-1} =& 2^{n-3} + (1-b)(s_{n-1} - b s_{n-2})
\nonumber
\\
b(s_{n-1} - b s_{n-2}) =& 2^{n-4} b  + b(1-b)(s_{n-2} - b s_{n-3})
\nonumber
\\
& \vdots\\
b^{n-3}(s_3 - b s_2) =&  b^{n-3}  + b^{n-3}(1-b)(s_2 - b s_1).
\nonumber
\end{align}
We sum these expressions to conclude
\begin{align}
\nonumber
s_n - b^{n-2} s_2 = \sum_{k=0}^{n-3} 2^k b^{n-3-k} + (1-b)(s_{n-1}  -b^{n-2} s_1).
\end{align}
We perform a similar computation and sum to obtain
\begin{align}
s_n - (1-b) s_{n-1} =& \sum_{k=0}^{n-3} 2^k b^{n-3-k} + b^{n-2} s_2  - (1-b)b^{n-2} s_1 
\nonumber
\\
(1-b)( s_{n-1} - (1-b) s_{n-2} )  =& (1-b)\sum_{k=0}^{n-2} 2^k b^{n-2-k} + (1-b)b^{n-3} s_2  - (1-b)^2 b^{n-3} s_1 
\nonumber
\\
& \vdots
\nonumber
\\
(1-b)^{n-3}( s_3 - (1-b) s_2 )  =& (1-b)^{n-3}  + (1-b)^{n-3}b s_2  - (1-b)^{n-2} b s_1 
\end{align}
This yields
\begin{align}
s_n - (1-b)^{n-2} s_2 =&[\sum_{k=0}^{n-3} 2^k b^{n-3-k} + (1-b)\sum_{k=0}^{n-2} 2^k b^{n-2-k} + (1-b)^{n-3}] 
\nonumber
\\
&+ s_2\sum_{k=1}^{n-2}b^k (1-b)^{n-2-k} + s_1\sum_{k=1}^{n-2}b^{n-1-k} (1-b)^k
\nonumber
\end{align}
and thus
\begin{align}
\nonumber
s_n \leq (n-2) 2^{n-2} + 2^{n-2} s_2 + 2^{n-1} s_1.
\end{align}
Plug this expression into \eqref{Inductive: a_n}. We infer 
\begin{align}
\nonumber
a_n \leq \frac{n-2}{2^{n-1}} + \frac{s_1}{2^{n-1}} + \frac{s_1}{2^{n-1}} + \frac{a_2}{2^{n-4}} +\frac{a_1}{2^{n-2}}.
\end{align}
Then
\begin{align}
\nonumber
\|\tilde{u}_n-\tilde{u}_{n+j}\|_{\bE_0(I)} \leq \|\tilde{u}_n-\tilde{u}_{n+1}\|_{\bE_0(I)}+\cdots \|\tilde{u}_{n+j-1}-\tilde{u}_{n+j}\|_{\bE_0(I)}
\end{align}
can be made arbitrarily small by taking $n$ large.
We  conclude that $(\tilde{u}_n)$ is Cauchy in $C(I; H^{r-1})\cap C^1(I; H^{r-2})$  and thus converges in this space. 

We denote the limit by $\tilde{u}\in C(I; H^{r-1})\cap C^1(I; H^{r-2})$. We can let 
$n\to \infty$ in \eqref{1st_order_sys_ind} and thus $\tilde{u}$ satisfies
\begin{align}
\cF(\tilde{u}) \tilde{u}   &= \tilde{\cR}(\tilde{u}), \quad \tilde{u}(0)=\tilde{u}_0.
\nonumber
\end{align}
Next, notice that it follows from \eqref{Bound EE} that
\begin{align}
\norm{\tilde{u}(t)}_r + \| \partial_t \tilde{u}(t)\|_{r-1} \leq \cC,\quad t\in [0,T].
\nonumber
\end{align}
We remark that since we have an estimate for the difference of two solutions,
uniqueness also follows from the above arguments.

\subsection{Continuity of solution}
The weak continuity of the solution $\tilde{u}$ can be proved by a similar argument to that of 
quasilinear wave equations, since in that proof the structure of the equation is not necessary but only the convergence $\tilde{u}_n\to \tilde{u} $ in $C(I; H^{r-1})\cap C^1(I; H^{r-2})$ and an estimate of the form~\eqref{Bound EE} matter.

We put
\begin{align}
\nonumber
\mathfrak{K}(t)=\sqrt{\frac{C_0}{2}I + (\mathfrak{B}(\tilde{u}(t)))^* \mathfrak{B}(\tilde{u}(t)) }
\end{align}
and
\begin{align}
\nonumber
\cA_r(t)=\cA_r(\tilde{u}(t))=  \mathfrak{K}(t) \la \nabla \ra ^r.
\end{align}
Hence
\begin{align}
\nonumber
N_r(t)=N_r(\tilde{u}(t)) = \cA_r(\tilde{u}(t))^*  \cA_r(\tilde{u}(t)).
\end{align}
Recall that $\mathfrak{B}\in OP \cS^0_0(r,2)$. 
It follows from \cite[Theorems~2.2 and 2.4]{MarschallPSOSobolev} that
\begin{align}
\label{S5: mapping K}
\mathfrak{K}(t) \in \cL(H^s),\quad -r<s<r-1.
\end{align}
Fix $t_0\in [0,T]$, let us first show that $\cA_r(t_0) \tilde{u}(t)$ is weakly continuous in $H^0$. Given any $\epsilon>0$ and $\phi\in H^0$, take a sequence of Schwarts functions $\phi_j\to \phi $ in $H^0$. Then
\begin{align}
&(\cA_r(t_0) \tilde{u}(t)- \cA_r(t_0) \tilde{u}_n(t), \phi) 
\nonumber
\\
=& (\cA_r(t_0) \tilde{u}(t)- \cA_r(t_0) \tilde{u}_n(t), \phi - \phi_j) 
+ (\cA_r(t_0) \tilde{u}(t)- \cA_r(t_0) \tilde{u}_n(t), \phi_j)
\nonumber
\end{align}
The first term is bounded by
\begin{align}
\nonumber
|(\cA_r(t_0) \tilde{u}(t)- \cA_r(t_0) \tilde{u}_n(t), \phi - \phi_j)| \leq \mathscr{K}(\cC) \|\phi - \phi_j\|_0
\end{align}
in view of \eqref{Bound EE}. By choosing $j$ large enough, we can make 
this term less than $\epsilon/2$.
Then fixing $j$ in the second term, we have
\begin{align}
\begin{split}
& |(\cA_r(t_0) \tilde{u}(t)- \cA_r(t_0) \tilde{u}_n(t), \phi_j)| \\
= & |( \la \nabla \ra ^{r-1} ( \tilde{u}(t)-   \tilde{u}_n(t)),  
\la \nabla \ra   \mathfrak{K}^*(t) \phi_j)|
\end{split}
\nonumber
\end{align}
Since $\tilde{u}_n\to \tilde{u} $ in $C(I; H^{r-1})$, taking
into consideration \cite[Theorem~2.4]{MarschallPSOSobolev} and \eqref{S5: mapping K}, we have
\begin{align}
\nonumber
|(\cA_r(t_0) \tilde{u}(t)- \cA_r(t_0) \tilde{u}_n(t), \phi_j)| <\epsilon/2
\end{align}
for all $n \geq n_0$ with some large enough $n_0$. In sum,
\begin{align}
\nonumber
|(\cA_r(t_0) \tilde{u}(t)- \cA_r(t_0) \tilde{u}_n(t), \phi) | <\epsilon \quad \text{for all }n\geq n_0 \text{ and } t\in [0,T].
\end{align}
This shows that $\cA_r(t_0) \tilde{u}_n(t) $ converges to $\cA_r(t_0) \tilde{u}(t)$ uniformly in $t$ in the weak topology. Thus,
$\cA_r(t_0) \tilde{u}(t)$ is weakly continuous in $t$ with respect to the norm of $H^0$.

In the next step, we will show $  u (t) \in C(I; H^r)$. In view of the weak continuity of $  u (t)$, it suffices to demonstrate that the map 
\begin{align}
\nonumber
[t\mapsto \|   \tilde{u} (t) \|_r] \quad \text{is continuous.}
\end{align}
Applying \eqref{EE inequality} to \eqref{1st_order_sys} and in view of \eqref{Bound EE}, we infer that
\begin{align}
\frac{d}{dt}\| \cA_r(t) \tilde{u}(t)\|^2_0  \leq \mathscr{K}(\cC).
\label{Basic EE 2}
\end{align}
This implies that
\begin{align}\label{Lip cont}
 \| \cA_r(t) \tilde{u}(t)\|^2_0 =:Y(t) \quad \text{is Lipschitz continuous in } t.
\end{align}
Now consider
\begin{align}
& \| \cA_r(t_0) \tilde{u}(t)\|^2_0 - \| \cA_r(t_0) \tilde{u}(t_0)\|^2_0  
\nonumber
\\
=&   (\| \cA_r(t_0) \tilde{u}(t)\|^2_0 - \| \cA_r(t) \tilde{u}(t)\|^2_0   ) 
+  (\| \cA_r(t) \tilde{u}(t)\|^2_0 - \| \cA_r(t_0) \tilde{u}(t_0)\|^2_0   ).
\nonumber
\end{align}
The first term on the RHS can be estimated as follows.
\begin{align}
& |(\| \cA_r(t_0) \tilde{u}(t)\|^2_0 - \| \cA_r(t) \tilde{u}(t)\|^2_0   )| 
\nonumber
\\
= &|\| \cA_r(t_0) \tilde{u}(t)\|_0 - \| \cA_r(t) \tilde{u}(t)\|_0   |(\| \cA_r(t_0) \tilde{u}(t)\|_0 + \| \cA_r(t) \tilde{u}(t)\|_0   )
\nonumber
\\
\leq & \mathscr{K}(\cC) \| (\mathfrak{K}(t)-\mathfrak{K}(t_0)) \la \nabla \ra ^r   \tilde{u}(t)\|_0  
\nonumber
 \\
\leq & \mathscr{K}(\cC) \| \mathfrak{K}(t)-\mathfrak{K}(t_0)\|_{\cL(H^0)}.
\nonumber
\end{align}
As elements in $\cL(H^0)$,  it is not hard to check that $\mathfrak{K}(\tilde{u})$ depends continuously on $\norm{\tilde{u}}_{r-1}$.
Combining with \eqref{Lip cont}, this observation shows 
that $[t \mapsto \|\cA_r(t_0) \tilde{u} (t)\|_0] $ is continuous at $t_0$; and thus 
\begin{center}
$\cA_r(t_0) \tilde{u} (t)$ is continuous in $t$ at $t_0$ w.r.t. $H^0$. 
\end{center}
Since $t_0$ is arbitrary, from
\begin{align}
\|\tilde{u}(t )-\tilde{u}(t_0)\|_r^2 \leq C \|\cA_r(t_0) \tilde{u} (t ) 
- \cA_r(t_0) \tilde{u} (t_0)\|_0^2,
\nonumber
\end{align}
we infer that $\tilde{u}   \in C(I; H^r)$.
Using this fact and equation~\eqref{1st_order_sys_ind}, we immediately conclude that
\begin{align}
\nonumber
\tilde{u}   \in C(I; H^r) \cap C^1(I, H^{r-1}).
\end{align}

\subsection{Solution to the original system} For Gevrey-regular data,
equations \eqref{E:u_unit} and \eqref{E:Div_T} admit a unique Gevrey-regular solution
\cite{BemficaDisconziNoronha,DisconziFollowupBemficaNoronha}. Solution
to \eqref{E:u_unit} and \eqref{E:Div_T} for Sobolev regular data, as in Theorem
\ref{T:main_theorem}, thus follows by a standard approximation argument, 
applying our energy estimates to the approximating Gevrey solutions.


\bibliography{References.bib}

\end{document}